\documentclass[12pt]{amsart}
\pdfoutput=1
\usepackage{amsfonts}
\usepackage{amssymb}
\usepackage{amsmath}
\usepackage{amsrefs}
\usepackage{mathrsfs}
\usepackage{proof}
\usepackage{fullpage}
\usepackage{stmaryrd}

\newtheorem{theorem}{Theorem}[section]
\newtheorem{lemma}[theorem]{Lemma}
\newtheorem{proposition}[theorem]{Proposition}
\newtheorem{corollary}[theorem]{Corollary}
\newtheorem{result}{Result}
\theoremstyle{definition}
\newtheorem{definition}[theorem]{Definition}

\newcommand{\B}{\mathcal B}
\renewcommand{\H}{\mathcal H}
\renewcommand{\L}{\mathcal L}

\newcommand{\Q}{\mathcal Q}
\newcommand{\X}{\mathcal X}
\newcommand{\Y}{\mathcal Y}
\newcommand{\Z}{\mathcal Z}
\newcommand{\qqq}{\mathbf q}
\renewcommand{\:}{\colon}
\newcommand{\inv}{^{-1}}
\newcommand{\subsetof}{\subseteq}
\newcommand{\suchthat}{\,|\,}
\renewcommand{\And}{\mathbin{\wedge}}
\newcommand{\Or}{\mathbin{\vee}}
\newcommand{\Not}{\neg}
\newcommand{\Implies}{\mathbin{\rightarrow}}
\newcommand{\Andthen}{\mathbin{\&}}
\newcommand{\Comp}{\mathbin{\bot\!\!\!\bot}}

\newcommand{\Proves}{\vdash}
\newcommand{\IM}{\quad\Longrightarrow\quad}
\newcommand{\EV}{\quad\Longleftrightarrow\quad}
\newcommand{\cat}{\mathbf}
\newcommand{\Rel}{\mathrm{Rel}}
\renewcommand{\[}{\llbracket\,}
\renewcommand{\]}{\,\rrbracket}

\def\drv#1#2{\infer{#2}{#1}}

\allowdisplaybreaks

\begin{document}

\title{A natural deduction system for orthomodular logic}
\author{Andre Kornell}
\address{Department of Computer Science, Tulane University, New Orleans, Louisiana 70118}
\email{akornell@tulane.edu}
\thanks{This work was supported by the AFOSR under MURI grant FA9550-16-1-0082.}

\begin{abstract}
Orthomodular logic is a weakening of quantum logic in the sense of Birkhoff and von Neumann. Orthomodular logic is shown to be a nonlinear noncommutative logic. Sequents are given a physically motivated semantics that is consistent with exactly one semantics for propositional formulas that use negation, conjunction, and implication. In particular, implication must be interpreted as the Sasaki arrow, which satisfies the deduction theorem in this logic. As an application, this deductive system is extended to two systems of predicate logic: the first is sound for Takeuti's quantum set theory, and the second is sound for a variant of Weaver's quantum logic.
\end{abstract}

\maketitle

\section{Introduction}\label{S1}

Quantum logic as it was first defined by Birkhoff and von Neumann \cite{BirkhoffvonNeumann} is a semantics for propositional formulas that use negation, conjunction, and disjunction. Propositional letters are interpreted as closed subspaces of a separable infinite-dimensional complex Hilbert space. The closed subspaces of a Hilbert space form a complete lattice that is equipped with an order-reversing involution that takes each element to a complement of that element. Meets interpret conjunction, joins interpret disjunction, and this involution, called orthocomplementation, interprets negation. The partial order then yields a notion of entailment for propositional formulas. Because the separable infinite-dimensional Hilbert space is unique up to isomorphism, this notion of entailment may be regarded as canonical.

This semantics for propositional formulas is well motivated physically. The physical system that consists of a single quantum particle may be modeled by a separable infinite-dimensional complex Hilbert space \cite{vonNeumann}, and it may be that every physical system can be modeled in this way \cite{JaffeWitten}. In this context, the closed subspaces of the Hilbert spaces model true-false-valued physical quantities that are measurable in principle. In short, the closed subspaces of the Hilbert space are propositions about the physical system. 

The inclusion of one closed subspace into another models the entailment of propositions in the fully literal sense that measuring the antecedent proposition to be true guarantees that a successive measurement of the consequent proposition will find it to be true as well. The orthocomplementation of closed subspaces likewise models the negation of propositions in the fully literal sense that measuring the negation of a proposition to be true guarantees that a successive measurement of that proposition will find it to be false. Unfortunately, there is no such simple account of conjunction and disjunction because measuring the truth-value of one proposition may alter the truth-value of another. Nevertheless, the conjunction of two propositions is the weakest proposition that entails both conjuncts.

The development of quantum logic has proceeded slowly by modern standards. Soon after the introduction of quantum logic, Husimi observed that this orthocomplemented lattice of propositions satisfies the orthomodular law \cite{Husimi}:
\begin{equation*} a \leq b \IM a \Or (\Not a \And b) = b,\end{equation*}
where $\Not a$ is the orthocomplement of $a$. It is equivalent to the validity of the entailment $\psi \Proves \phi \Or(\Not \phi \And (\phi \Or \psi)))$. Those orthocomplemented lattices that satisfy the orthomodular law are called orthomodular lattices \cite{Holland}. The first entailment that is valid in quantum logic but not in orthomodular lattices was discovered by Greechie \cite{Greechie}. A number of other examples were then found and investigated \cite{Godowski}\cite{Mayet}\cite{MegillPavicic}\cite{Mayet2}\cite{PavicicMegill2}\cite{Mayet3}\cite{MegillPavicic2}\cite{MegillPavicic3}. Recently, building on the work of Slofstra \cite{Slofstra}, Fritz showed that the first-order theory of the ortholattice of closed subspaces of a separable infinite-dimensional Hilbert space is undecidable \cite{Fritz}. It is apparently still unknown whether the validity of an entailment in quantum logic is decidable.

This paper introduces a deductive system for orthomodular logic, which is also known as orthomodular quantum logic and sometimes simply as quantum logic. The valid entailments of orthomodular logic are those inequalities that hold universally in all orthomodular lattices. Orthomodular logic may be regarded as a pragmatic approximation to quantum logic or, more speculatively, as a generalization of quantum logic that is applicable to physical models with exotic scalars such as formal Laurent series \cite{GrossKunzi}\cite{KellerOchsenius}. An abundance of deductive systems for orthomodular logic have been proposed \cite{Kalmbach}\cite{Dishkant}\cite{Mittelstaedt2}\cite{Goldblatt}\cite{DallaChiara}\cite{Zeman}\cite{Nishimura}\cite{Hardegree1}\cite{CutlandGibbins}\cite{Pavicic}\cite{DelmasRigoutsos}\cite{PavicicMegill}\cite{CattaneoDallaChiaraGiuntiniPaoli}\cite{Titani}; most are glossed in \cite{PavicicMegill3}*{sec.~1} along with other systems of derivation. These deductive systems range from Hilbert-style systems such as \cite{Kalmbach} and \cite{DallaChiara} to Gentzen-style systems such as \cite{Nishimura} and \cite{DelmasRigoutsos}.
What is the use of yet another deductive system?

Quantum logic has been the subject of criticism on several fronts: that it is defective as a logic, that it is conceptually misguided, and that it is useless. The first of these critiques often focuses on the absence of a well-behaved implication that satisfies the deduction theorem, e.g., ``there is no obvious notion of implication or deduction''\cite{AbramskyCoecke}, ``no satisfactory implication operator has been found (so that there is no deductive system in quantum logic)''\cite{HeunenLandsmanSpitters}. The former authors go on to observe that ``quantum logic was therefore always seen as logically very weak, or even a non-logic.'' The second critique is more philosophical and has reached extreme conclusions, e.g., ``the tale of quantum logic is not the tale of a promising idea gone bad, it is rather the tale of the unrelenting pursuit of a bad idea''\cite{Maudlin}, ``quantum logic was a mistake \textit{a priori}''\cite{Girard}. The third critique is simply that quantum logic has not played its hoped-for role in either physics or in mathematics, e.g., ``quantum logic ...\ has never yielded the goods''\cite{Maudlin}, ``it has proved to be \emph{mathematically sterile}: it fails to link up in interesting ways with mainstream developments in mathematical physics''\cite{Halvorson}.

This paper offers some evidence to the contrary for each of these three critiques. First, it introduces a natural deductive system for orthomodular logic that does carry a canonical notion of implication that satisfies the deduction theorem. Second, it introduces a natural semantics for sequents that mirrors experimental procedure. Third, it generalizes this deductive system to a quantum predicate logic that is applicable to discrete quantum structures in the sense of noncommutative geometry. This is the application that motivated this research.

The immediate subject of this paper is the following deductive system:

\begin{definition}\label{1.1}
We define the propositional deductive system $NOM$:
\begin{enumerate}
\item a \emph{formula} is built up from propositional letters using the connectives $\And$, $\Implies$, and $\neg$;
\item a \emph{sequent} is an expression of the form $\phi_1, \ldots, \phi_n \Proves \psi$, where $\phi_1 \ldots, \phi_n$, and $\psi$ are formulas with $n \geq 0$;
\item a sequent is \emph{derivable} if it can be derived using the inference rules in Figure~\ref{F1}, where $\phi$, $\psi$, and $\chi$ are formulas and $\Gamma$ is a finite sequence of formulas, possibly empty.
\end{enumerate}
\end{definition}

\begin{figure}[h]

\begin{equation*}
\drv
{}
{\Gamma, \phi \Proves \phi}
\end{equation*}

\begin{equation*}
\drv
{\Gamma \Proves \phi & \Gamma, \phi \Proves \psi}
{\Gamma \Proves \psi}
\qquad\qquad
\drv
{\Gamma \Proves \phi & \Gamma \Proves \psi}
{\Gamma, \phi \Proves \psi}
\end{equation*}

\begin{equation*}
\drv
{\Gamma, \phi, \psi \Proves \phi &\Gamma, \phi, \psi \Proves \chi &   \Gamma, \psi, \phi \Proves \psi}
{\Gamma, \psi, \phi \Proves \chi}
\end{equation*}

\begin{equation*}
\drv
{\Gamma \Proves \phi & \Gamma \Proves \psi}
{\Gamma \Proves \phi \And \psi}
\qquad\qquad
\drv
{\Gamma \Proves \phi \And \psi}
{\Gamma \Proves \phi}
\qquad\qquad
\drv
{\Gamma \Proves \phi \And \psi}
{\Gamma \Proves \psi}
\end{equation*}

\begin{equation*}
\drv
{\Gamma, \phi \Proves \psi}
{\Gamma \Proves \phi \Implies \psi}
\qquad \qquad
\drv
{\Gamma \Proves \phi \Implies \psi}
{\Gamma, \phi \Proves \psi}
\end{equation*}

\begin{equation*}
\drv
{\Gamma, \phi \Proves \psi & \Gamma, \Not \phi \Proves \psi}
{\Gamma \Proves \psi}
\qquad \qquad
\drv
{\Gamma \Proves \Not \phi}
{\Gamma, \phi \Proves \psi}
\end{equation*}
\caption{The deductive system $NOM$.}\label{F1}
\end{figure}

\noindent
The antecedent of a sequent is a sequence of formulas, not a set. The intuitive meaning of a judgement $\phi_1, \ldots, \phi_n \Proves \psi$ is that after verifying each of the formulas $\phi_1, \ldots, \phi_n$ via measurement and in that order, a measurement of the truth-value of $\psi$ is guaranteed to find that it is true. The semantics that we will shortly define is entirely faithful to this intuition.

The deductive system $NOM$ includes an implication connective that satisfies the deduction theorem by fiat. Furthermore, it uniquely determines a notion of implication in orthomodular logic:

\begin{result}[Sasaki arrow]\label{R1}
The following rules of inference are derivable in $NOM$:
\begin{equation*}
\begin{aligned}
\drv
{\Gamma \Proves \phi \Implies \psi}
{\Gamma \Proves \Not (\phi \And \Not (\phi \And \psi))}
\end{aligned}\;, \qquad \qquad
\begin{aligned}
\drv
{\Gamma \Proves \Not (\phi \And \Not (\phi \And \psi))}
{\Gamma \Proves \phi \Implies \psi}
\end{aligned}\;.
\end{equation*}
\end{result}

\noindent This implication connective is of course the familiar Sasaki arrow \cite{HermanMarsdenPiziak}\cite{Sasaki}. A number of authors have argued that the Sasaki arrow is the natural implication connective for quantum logic \cite{Kunsemuller}\cite{Finch}\cite{Mittelstaedt}\cite{Mittelstaedt2}\cite{Hardegree}\cite{Hardegree2}. Result~\ref{R1} motivates the Sasaki arrow purely from logical considerations; we have made no assumptions on the implication connective except those that are implicit in the notion itself; c.f.~\cite{Mittelstaedt3}*{p.~187}.

The deductive system $NOM$ might be called a substructural logic because its logical rules are all familiar rules of inference from classical logic. The weakening and contraction rules
\begin{equation*}
\begin{aligned}
\drv
{\Gamma \Proves \psi}
{\phi, \Gamma \Proves \psi}
\end{aligned},
\qquad \qquad
\begin{aligned}
\drv
{\Gamma, \phi, \phi \Proves \chi}
{\Gamma, \phi \Proves \chi}
\end{aligned}
\end{equation*}
are both admissible in $NOM$, but the exchange rule
\begin{equation*}
\begin{aligned}
\drv
{\Gamma, \phi, \psi \Proves \chi}
{\Gamma, \psi, \phi \Proves \chi}
\end{aligned}
\end{equation*}
is not admissible. The addition of this exchange rule recovers classical logic:
\begin{result}[noncommutativity]\label{R2}
Let $\phi_1, \ldots, \phi_n$, and $\psi$ be propositional formulas in the connectives $\And$, $\Implies$, and $\Not$. Then, the following are equivalent:
\begin{enumerate}
\item the sequent $\phi_1, \ldots, \phi_n \Proves \psi$ is derivable in $NOM$ with the exchange rule;
\item the formula $(\phi_1 \And \cdots \And \phi_n) \to \psi$ is classically derivable.
\end{enumerate}
\end{result}
\noindent The deductive system $NOM$ is not formally a substructural logic \cite{Restall} because the cut rule that is standard in substructural logics is not admissible. Nevertheless, Result~\ref{R2} justifies the conceptual claim that $NOM$ is a noncommutative logic.

The idea that orthomodular logic could be a noncommutative logic is natural and even expected due to the prominent role of noncommutativity in quantum theory. Noncommutativity is a key feature in the standard general formulation of Heisenberg's uncertainty principle \cite{Griffiths}. It is also the defining feature of quantum mathematics in the sense noncommutative geometry \cite{GraciaBondiaVarillyFigueroa}. Furthermore, Malinowski proved in essence that no commutative formulation of orthomodular logic can satisfy the deduction theorem \cite{Malinowski}*{Thm.~2.7}, though a ``semiclassical'' deductive theorem may be formulated \cite{Titani}*{p.~668}.

It is therefore surprising that no noncommutative deductive systems for orthomodular logic have appeared in the literature. This is all the more surprising because the idea that the order of assumptions should be significant to quantum logic has been investigated before. In this respect, the nearest proof system to $NOM$ is probably Mittelstaedt's quantum dialogic \cite{Mittelstaedt3}. This is essentially a game-theoretic semantics that formalizes a debate between a proponent and an opponent of a proposition. The players are constrained in when they may cite a previously proved proposition, and this constraint reflects that the result of one measurement may be ``destroyed'' by another \cite{Mittelstaedt3}*{p.~178}. The same intuition explains why all sequents of the form $\Gamma, \phi, \psi \Proves \psi$ are derivable in $NOM$, but not all sequents of the form $\Gamma, \psi, \phi \Proves \psi$ are derivable. Surprisingly, the natural deduction system that has been obtained from quantum dialogic is commutative \cite{Stachow}.

A noncommutative semantics for sequents was previously considered by Bell, and furthermore, his semantics satisfies the deduction theorem \cite{Bell}*{eq.~4.5}. However, the deductive system $NOM$ is not sound for Bell's semantics \cite{Bell}*{p.~95}, and Bell describes the specification of a deductive system for his semantics as an ``open problem'' \cite{Bell}*{p.~97}.

We define a semantics for $NOM$ that interprets each sequent of experimentally decidable propositions as an experimentally falsifiable proposition. For any interpretation of formulas of any kind as true-false-valued observables on a physical system, the naive meaning of a judgement $\phi_1, \ldots, \phi_n \Proves \psi$ is that whenever we have verified the formulas $\phi_1, \ldots, \phi_n$, we may also be certain of the truth of $\psi$. If we formalize these observables by closed subspaces of a Hilbert space and we notate our interpretation by $\[\cdot\]$, then this condition is equivalent to the inclusion $\[\phi_1\] \Andthen \cdots \Andthen \[\phi_n\] \leq \[ \psi\]$, where $\Andthen$ is the Sasaki projection \cite{Sasaki}. This equivalence appears to be folklore; we record a proof of it in Appendix~\ref{A}.

If $A$ and $B$ are closed subspaces of a Hilbert space, then $A \Andthen B$ may be defined as the closure of the projection of $A$ onto $B$. Thus, $A \Andthen B$ is spanned by those vector states that may be obtained by verifying $A$ and then verifying $B$, where the term ``verify'' means ``measure to be true''. The four structural rules of $NOM$ are valid for this interpretation of sequents in any physical system that can modeled using a Hilbert space, in contrast to the structural rules of classical natural deduction. The seven logical rules of $NOM$ simply express the familiar meanings of the logical connectives $\And$, $\Implies$, and $\Not$.

We may formalize this discussion. In any orthomodular lattice, we may define the Sasaki projection of elements $a$ and $b$ to be the element $a \Andthen b = (a \Or \Not b) \And b$. Our convention is that the connective $\Andthen$ associates to the left. Using this definition, we may interpret sequents in any orthomodular lattice, and the deductive system $NOM$ is sound and complete for the resulting semantics:

\begin{result}[soundness]\label{R3}
Let $\Phi$ be the set of formulas of $NOM$, let $\Q$ be an orthomodular lattice, and let $\[\cdot\]\: \Phi \to \Q$ be any surjective function. The following are equivalent:
\begin{enumerate}
\item the set of all sequents $\phi_1, \ldots, \phi_n \Proves \psi_0$ such that $\[\phi_1\] \Andthen \cdots \Andthen \[\phi_n\] \leq \[\psi_0\]$ is closed under the rules of inference of $NOM$;
\item the function $\[\cdot\]$ satisfies
\begin{enumerate}
\item $\[\phi \And \psi \] = \[\phi\] \And \[\psi\]$,
\item $\[\phi \Implies \psi\] = \Not \[\phi\] \Or (\[\phi\] \And \[\psi\])$,
\item $\[\Not \phi\] = \Not \[\phi\]$.
\end{enumerate}
\end{enumerate}
\end{result}\label{R4}
\begin{result}[completeness]
Let $\phi_1, \ldots, \phi_n \Proves \psi_0$ be a sequent of $NOM$. Assume that the equation $\[\phi_1\] \Andthen \cdots \Andthen \[\phi_n\] \leq \[\psi_0\]$ holds for every function $\[\cdot\]: \Phi \to \Q$ such that $\Q$ is an orthomodular lattice and such that $\[\cdot\]$ satisfies equations \emph{(a)\ndash(c)} in the statement of Result~\ref{R3}. Then, the sequent $\phi_1, \ldots, \phi_n \Proves \psi_0$ is derivable in $NOM$.
\end{result}

To simplify this discussion in summary, let us assume that all experimentally decidable propositions may be modeled by closed subspaces of a Hilbert space. Then, the meaning that we ascribe to sequents is their literal empirical reading. The logic of such sequents is \emph{a~priori} the logic of the physical universe. Furthermore, the connectives $\And$, $\Implies$, and $\Not$ can be interpreted to obey the same primitive rules that they obey classically, and this uniquely determines their interpretation. In particular, $\Implies$ must be interpreted as the Sasaki arrow, and it does satisfy the deduction theorem. In this way, we may begin to address the first and second cited critiques of quantum logic.

The third critique that quantum logic has been useless is the strongest critique of the three, but it is also the least conclusive due to its empirical nature. Indeed, none of the cited deductive systems have yet been used to obtain new physical predictions or a new perspective on quantum phenomena, and our deductive system is no different in this regard. That may still change. Quantum phenomena are predominantly probabilistic in character, and we may reasonably expect that their account requires not a quantum logic, but a quantum probability theory. In its essence, this view predates quantum logic \cite{MurrayvonNeumann}. Classical probability theory may be formalized within classical logic by means of probability spaces, and by analogy, quantum probability theory might be formalized within quantum logic. It is unsurprising that such a formalization has not been obtained yet; substantial foundational development separates the specification of classical logic and the definition of probability spaces within that logic.

Nevertheless, this paper offers some evidence that quantum logic or, more accurately, orthomodular logic is both usable and applicable. By usablility, we mean the availability of intelligible derivations that reflect a coherent mode of reasoning as in the case of intuitionistic logic. To make this evidence explicit, we include derivations that might otherwise be left as exercises for the reader. In particular, we continue the development of the deductive system $NOM$ by deriving the primitive rules of inference for the defined connectives $\Or$ and $\Comp$:

\begin{result}[disjunction and compatibility]\label{R5}
Let $\phi \Or \psi$ be an abbreviation for $\Not (\Not \phi \And \Not \psi)$, and let $\phi \Comp \psi$ be an abbreviation for $(\phi \Implies (\psi \Implies \phi)) \And (\psi \Implies (\phi \Implies \psi))$. Then, the rules of inference in Figure~\ref{F2} are derivable in $NOM$.
\end{result}

\begin{figure}[h]

\begin{equation*}
\drv
{\Gamma \Proves \phi}
{\Gamma \Proves \phi \Or \psi}
\qquad\qquad
\drv
{\Gamma \Proves \psi}
{\Gamma \Proves \phi \Or \psi}
\end{equation*}

\begin{equation*}
\drv
{\Gamma \Proves \phi \Or \psi & \Gamma, \phi \Proves \chi & \Gamma, \psi \Proves \chi & \Gamma, \chi, \phi \Proves \chi & \Gamma, \chi, \psi \Proves \chi}
{\Gamma \Proves \chi}
\end{equation*}

\begin{equation*}
\drv
{\Gamma, \phi, \psi \Proves \phi & \Gamma, \psi, \phi \Proves \psi}
{\Gamma \Proves \phi \Comp \psi}
\end{equation*}

\begin{equation*}
\drv
{\Gamma \Proves \phi \Comp \psi & \Gamma, \phi, \psi \Proves \chi}
{\Gamma, \psi, \phi \Proves \chi}
\qquad\qquad
\drv
{\Gamma \Proves \phi \Comp \psi & \Gamma, \psi, \phi \Proves \chi}
{\Gamma, \phi, \psi \Proves \chi}
\end{equation*}
\caption{Primitive rules of inference for $\Or$ and $\Comp$.}\label{F2}
\end{figure}

We also exhibit a derivation that the compatibility connective $\Comp$ is equivalent to the commutator that was introduced by Marsden \cite{Marsden}. This notation is due to Takeuti \cite{Takeuti}*{sec.~1}.

\begin{result}[commutator]\label{R6}
The following rules of inference are derivable in $NOM$:
\begin{equation*}
\begin{aligned}
\drv
{\Gamma \Proves \phi \Comp \psi}
{\Gamma \Proves ((\phi \And \psi) \Or (\phi \And \Not \psi)) \Or ((\Not\phi \And \psi) \Or (\Not\phi \And \Not \psi))}
\end{aligned}\;,
\end{equation*}
\begin{equation*}
\begin{aligned}
\drv
{\Gamma \Proves ((\phi \And \psi) \Or (\phi \And \Not \psi)) \Or ((\Not\phi \And \psi) \Or (\Not\phi \And \Not \psi))}
{\Gamma \Proves \phi \Comp \psi}
\end{aligned}\;.
\end{equation*}
\end{result}

\noindent The primitive rules of inference given for $\Comp$ in Figure~\ref{F2} clearly imply that $\phi \Comp \psi$ is equivalent to $(\phi \Implies (\psi \Implies \phi)) \And (\psi \Implies (\phi \Implies \psi))$, and hence, Result~\ref{R6} is comparable to Result~\ref{R1} in that it motivates the common interpretation of a connective from the primitive inference rules for that connective.

To provide some evidence of the applicability of the deductive system $NOM$, we extend it to a deductive system for quantum predicate logic. We do so in two distinct ways because there are broadly two distinct notions of a quantum set. These two notions do not represent competing conceptions of quantum mathematics but rather distinct generalizations of sets within the same conception of quantum mathematics.

The term ``quantum set theory'' commonly refers to a generalization of the Boolean-valued models $V^{(\B)}$ in which the complete Boolean algebra $\B$ is replaced with a complete orthomodular lattice $\Q$ \cite{Takeuti}\cite{Ozawa}\cite{Ozawa4}\cite{Ozawa5}\cite{DoringEvaOzawa}. In the case of primary interest, this complete orthomodular lattice $\Q$ consists of the closed subspaces of a Hilbert space $\H$. By Gleason's theorem, the states of a physical system modeled by $\H$ are in one-to-one correspondence with countably additive measures on $\Q$, as long as the dimension of $\H$ is greater than two \cite{Gleason}. In this way, each state assigns a probability to each sentence of the language of set theory with constants from the orthomodular-valued model $V^{(\Q)}$. The canonical name of a generic ultrafilter behaves in some ways like a ``hidden variable'' of the physical system.

\begin{result}[Takeuti's semantics]\label{R7}
Let $\Q$ be a complete orthomodular lattice. Let $\[\cdot\]_\Q$ be Takeuti's interpretation of formulas in the language of set theory with constants from $V^{(\Q)}$ \cite{Takeuti}\cite{Ozawa5}. Let $NOM_\Q$ be the deductive system of single-sorted predicate logic whose rules of inference are those of Figure~\ref{F1} together with 
\begin{equation*}
\begin{aligned}
\drv
{\Gamma \Proves \phi[x/y]}
{\Gamma \Proves (\forall x)\phi}
\end{aligned}\;,
\qquad \qquad
\begin{aligned}
\drv
{\Gamma \Proves (\forall x)\phi}
{\Gamma \Proves \phi[x/t]}
\end{aligned}\;,
\end{equation*}
where the former rule is subject to the standard constraint that $y$ must not appear freely in $\Gamma \Proves (\forall x)\phi$. Then, for each closed formula $\psi$, if $\Proves \psi$ is derivable in $NOM_\Q$, then $\[\psi\]_\Q = \top$.
\end{result}

The most prominent connection between quantum set theory and quantum mechanics is the so-called Takeuti correspondence, which is a canonical bijection between the real numbers in $V^{(\Q)}$ and the real-valued observables on a physical system modeled by $\H$, when $\Q$ consists of the closed subspaces of $\H$ \cite{Takeuti}*{p.~321}\cite{Ozawa}*{Thm.~6.1}. In this context, Ozawa has analyzed the compatibility of real-valued observables from the perspective of quantum logic \cite{Ozawa2}\cite{Ozawa3}.
We do not treat bounded quantifiers because they are not expressible within the language of set theory in quantum set theory \cite{Takeuti}*{p.~315}. However, if we treat quantum sets extensionally rather than intentionally, then we may clearly treat bounded quantifiers as abbreviations in the obvious way.

The term ``quantum set'' may also refer to the discrete quantum spaces of noncommutative geometry \cite{quantumsets}\cite{PodlesWoronowicz}. A quantum set in this sense is essentially just a von Neumann algebra that may be noncommutative but that is otherwise very similar to the von Neumann algebra of all bounded complex-valued functions on a set. 
We show that $NOM$ extends to a natural deductive system that is sound for Weaver's quantum predicate logic \cite{Weaver}*{sec.~2.6} as it applies to these quantum sets, albeit only for nonduplicating formulas. A formula of first-order logic is defined to be \emph{nonduplicating} if no variable occurs more than once in any atomic subformula. This syntactic constraint reflects the absence of diagonal functions in noncommutative geometry \cite{Woronowicz} and the impossibility of broadcasting quantum states \cite{BarnumCavesFuchsJozsaSchumacher}.

\begin{result}[Weaver's semantics]\label{R8}
Let $\[\cdot\]_\qqq$ be the semantics that is defined in \cite{discretequantumstructures}. Let $NOM_\qqq$ be the deductive system of many-sorted predicate logic whose rules of inference are those of Figure~\ref{F1} together with the following:
\begin{enumerate}
\item $
\begin{aligned}
\drv
{\Gamma \Proves \phi[x/y]}
{\Gamma \Proves (\forall x)\phi}
\end{aligned}\;,$
where $y$ does not appear freely in $\Gamma \Proves (\forall x) \phi$,
\item
$
\begin{aligned}
\drv
{\Gamma \Proves (\forall x)\phi}
{\Gamma \Proves \phi[x/t]}
\end{aligned}\;,
$ where $\phi$ and $t$ have no free variables in common,
\item
$
\begin{aligned}
\drv
{\Gamma, \phi, \psi \Proves \chi}
{\Gamma, \psi, \phi \Proves \chi}
\end{aligned}\;,
$ where $\phi$ and $\psi$ have no free variables in common.
\end{enumerate}
Then, for each closed formula $\psi$, if $\Proves \psi$ is derivable in $NOM_\qqq$, then $\[\psi\]_\qqq = \top$.
\end{result}

Many classes of discrete quantum structures may be axiomatized within quantum predicate logic \cite{discretequantumstructures}*{sec.~1.3} including discrete quantum graphs, discrete quantum groups, and discrete quantum isomorphisms. Such discrete quantum structures are quantum sets equipped with suitably generalized relations and functions of various arity. The ubiquitous qualifier ``discrete'' addresses the fact that within noncommutative geometry many structures are implicitly topological or measurable. For example, quantum groups generalize locally compact groups \cite{KustermansVaes}, and quantum graphs generalize measurable graphs \cite{Weaver2}. These discrete quantum structures all originate in mathematical physics. Quantum graphs were motivated by quantum error correction \cite{DuanSeveriniWinter}, quantum groups were motivated by the quantization of physical symmetry \cite{PodlesWoronowicz}, and quantum graph isomorphisms were motivated by quantum nonlocality \cite{MancinskaRoberson}\cite{AtseriasMancinskaRobersonSamalSeveriniVarvitsiotis}. The general notion of a discrete quantum structure also traces its origins to the problem of quantum error correction \cite{KuperbergWeaver}\cite{Weaver2}\cite{quantumsets}\cite{discretequantumstructures}. With these examples in hand, we may begin to address the third cited critique of quantum logic that it has no meaningful connection to physics or to mathematics.

The deductive systems $NOM_\Q$ and $NOM_\qqq$ are incomparable, even if we restrict $NOM_\qqq$ to a single sort. This is to be expected because we are generalizing the set-theoretic universe in two very different ways. Intuitively, a Boolean-valued model is a classical space whose points are copies of the cumulative hierarchy, even if the Boolean algebra has no complete ultrafilters and thus the space has no points in the literal sense. An orthomodular-valued model is similarly a quantum space whose points are copies of the cumulative heirarchy, but its universe, i.e., its class of constants, is entirely \emph{classical}. Each member of the universe is a mathematical object, which may be named and duplicated. In contrast, the universe of a single-sorted structure in noncommutative geometry is a single discrete \emph{quantum} space, whose points are just a figure of speech. In short, intuitively, $NOM_\Q$ is the logic of a quantum space of classical structures, whereas $NOM_\qqq$ is the logic of a classical space of quantum structures. For this reason, the closed formulas of $NOM_\qqq$ satisfy classical logic; this is reflected in rule (3) of Result~\ref{R8}. One convenient consequence of this principle is that a theorem that follows from a sequence of axioms may be asserted after any sequence of additional assumptions, because the theorem is a closed formula and it therefore commutes with each of the assumptions.

This paper does not include a completeness theorem that is converse to Result~\ref{R7} or Result~\ref{R8}. A variant of Titani's completeness theorem \cite{Titani}*{Thm.~28} might be provable, but it uses the axioms of quantum set theory in an essential way. Takeuti's quantum set theory is undergoing a reformulation \cite{Ozawa5}, so recording such a completeness theorem may be premature. A general completeness result for $NOM_\Q$ appears to be out of reach, because it is unknown whether every orthomodular lattice has a completion \cite{Harding}*{sec.~6}. A completeness result for $NOM_\qqq$ appears to be even further out of reach because orthomodular logic is far from being complete for finite-dimensional Hilbert spaces. A symmetric monoidal category of orthomodular quantum sets \cite{Rump} along the lines of \cite{quantumsets} might bridge that gap.

Finally, we define the existential quantifier and derive its primitive rules of inference:

\begin{result}[existential quantification]\label{R9}
Let $(\exists x)\phi$ be an abbrevation for $\Not(\forall x)\Not \phi$. Then, the following rules of inference are derivable in $NOM_\Q$ and $NOM_\qqq$:
\begin{equation*}
\begin{aligned}
\drv
{\Gamma \Proves \phi[x/t]}
{\Gamma \Proves (\exists x)\phi}
\end{aligned}\;,
\qquad \qquad
\begin{aligned}
\drv
{\Gamma \Proves (\exists x)\phi & \Gamma, \phi[x/y] \Proves \psi & \Gamma, \psi, \phi[x/y] \Proves \psi }
{\Gamma \Proves \psi}
\end{aligned}\;,
\end{equation*}
where the latter rule is subject to the standard constraint that $y$ must not appear freely in $\Gamma \Proves (\exists x)\phi$ or in $\psi$. In the case of $NOM_\qqq$, the former rule is also subject to the constraint that $t$ and $\phi$ must have no free variables in common.
\end{result}

\noindent We offer an intuitive justification of the new premise in the $\exists$-elimination rule. The assumption $\phi[x/y]$ may be regarded as creating an element $y$ that satisfies $\phi$. If $\Gamma \Proves (\exists x) \phi$ is derivable, then we may create such an element consistently: if $\Gamma \Proves (\exists x) \phi$ and $\Gamma, \phi[x/y] \Proves \psi \And \Not \psi$ are both derivable, then so is $\Gamma \Proves \psi \And \Not \psi$. However, the formula $\psi$ may be a consequence of creating $y$ and not a consequence of the possibility of creating $y$. In this sense, the existential quantifier expresses a potential rather than an actual existence. Modulo $\Gamma, \phi[x/y] \Proves \psi$, the sequent $\Gamma, \psi, \phi[x/y] \Proves \psi$ expresses that $\psi$ is compatible with the creation of $y$.

The first deductive system for quantum predicate logic was introduced by Dishkant \cite{Dishkant}. His inference rules for the existential quantifier,
\begin{equation*}
\begin{aligned}
\drv
{\phantom{\Proves \phi[x/y] \Implies_D \psi}}
{\Proves \phi[x/y] \Implies_D (\exists x) \phi  }
\end{aligned}\;,
\qquad \qquad
\begin{aligned}
\drv
{\Proves \phi[x/y] \Implies_D \psi}
{\Proves (\exists x) \phi \Implies_D \psi},
\end{aligned}\
\end{equation*}
where $\phi \Implies_D \psi$ abbreviates $\Not \psi \Implies \Not \phi$, are easily derived in $NOM_\Q$ by using the primitive rules for the universal quantifier. An equivalent deductive system was then investigated by Dunn, who showed that within quantum predicate logic, the axioms of Peano arithmetic imply those of classical logic \cite{Dunn}. More recently, Titani has formulated a sequent calculus for quantum predicate logic \cite{Titani}*{sec.~3.2}, which use a modally closed implication connective that satisfies a restricted deduction theorem.

We close this introduction with the remark that although $NOM$ is presented as a calculus of sequents, it is in substance a natural deduction system rather than a sequent calculus if we use Gentzen's systems $NK$ and $LK$ as a reference \cite{Gentzen}. We argue that the characteristic features of a natural deduction system are that its elimination rules never place a logical symbol into the antecedent and that each sequent has exactly one formula in its succedent. A new feature of $NOM$ relative to $NK$ is that some rules of inference allow the inference of a formula only after making an additional assumption, e.g., the $\Implies$-elimination rule. However, we argue that in the context of a noncommutative deductive system, this rule is \emph{more} faithful to the intuitive meaning of the implication connective than modus ponens. In $NOM$, an implication $\phi \Implies \psi$ expresses that you may assert $\psi$ after assuming $\phi$ and nothing more than that. Nevertheless, modus ponens is derivable in $NOM$; see Proposition~\ref{2.1}.

\subsection*{Conventions}

Top to bottom and left to right in Figure~\ref{F1}, the inference rules of $NOM$ are called the assumption rule, the cut rule, the paste rule, the compatible exchange rule, the $\And$-introduction rule, the left $\And$-elimination rule, the right $\And$-elimination rule, the $\Implies$-introduction rule, the $\Implies$-elimination rule, the excluded middle rule, and the deductive explosion rule. The unqualified terms ``formula'', ``sequent'', and ``derivation'' refer to the formulas, sequents, and derivations of $NOM$. Hilbert spaces are over the complex numbers throughout.

\section{Conjunction, implication, and negation}\label{S2}

This section exhibits a number of derivations, beginning with a few simple derivations of some familiar rules of inference and ending with two derivations that characterize the implication connective as the Sasaki arrow.

\begin{proposition}\label{2.1}
The following inference rule is derivable:
\begin{equation*}\begin{aligned}
\drv
{\Gamma \Proves \phi & \Gamma \Proves \phi \Implies \psi}
{\Gamma \Proves \psi}
\end{aligned}\,.
\end{equation*}
\end{proposition}

\begin{proof}
\begin{equation*}\begin{aligned}
\drv
{\Gamma \Proves \phi & \drv{\Gamma \Proves \phi \Implies \psi}{\Gamma, \phi \Proves \psi}}
{\Gamma \Proves \psi.}
\end{aligned}\
\end{equation*}
\end{proof}

\begin{proposition}\label{2.2}
The following inference rule is derivable:
\begin{equation*}\begin{aligned}
\drv
{\Gamma \Proves \phi & \Gamma \Proves \Not \phi}
{\Gamma \Proves \psi}
\end{aligned}\,.
\end{equation*}
\end{proposition}

\begin{proof}
\begin{equation*}
\drv{
\Gamma \Proves \Not \phi &
\drv{
\drv{
\Gamma \Proves \Not \phi & \Gamma \Proves \phi}{
\Gamma, \Not \phi \Proves \phi}
&
\drv
{\drv
{}
{\Gamma, \Not \phi \Proves \Not \phi}}
{\Gamma, \Not \phi, \phi \Proves \psi}}{
\Gamma, \Not \phi \Proves \psi}}{
\Gamma\Proves \psi.}
\end{equation*}
\end{proof}

\begin{lemma}\label{2.3}
The following sequents are derivable:
\begin{enumerate}
\item $\Gamma, \Not \phi, \phi \Proves \psi$,
\item $\Gamma, \Not \Not \phi \Proves \phi$,
\item $\Gamma, \phi, \Not \phi \Proves \psi$,
\item $\Gamma, \phi \Proves \Not \Not \phi$.
\end{enumerate}
\end{lemma}

\begin{proof}
\begin{equation*}
\begin{aligned}
\drv
{\drv{}{\Gamma , \Not \phi \Proves \Not \phi}}
{\Gamma, \Not \phi, \phi \Proves \psi;}
\end{aligned}
\qquad \qquad
\begin{aligned}
\drv{
\drv{}{
\Gamma, \Not \Not \phi, \phi \Proves  \phi}
&
\drv{}{
\Gamma, \Not \Not \phi,\Not \phi \Proves \phi}}{
\Gamma, \Not \Not \phi \Proves \phi.}
\end{aligned}
\end{equation*}

\quad

\begin{equation*}
\begin{aligned}
\drv{
\drv{
\drv{
\drv{
\drv{}{\Gamma, \Not \phi, \phi \Proves \Not \phi \Implies \psi}}
{\Gamma, \Not \phi \Proves \phi \Implies(\Not \phi \Implies \psi)}
&
\drv
{\drv
{\drv
{}
{\Gamma, \Not \Not \phi \Proves \phi}
&\drv
{\drv
{}
{\Gamma, \Not \Not \phi, \Not \phi \Proves \psi}}
{\Gamma, \Not \Not \phi \Proves \Not \phi \Implies \psi}}
{\Gamma, \Not \Not \phi, \phi \Proves \Not \phi \Implies \psi}}
{\Gamma, \Not \Not \phi \Proves \phi \Implies(\Not \phi \Implies \psi)}}
{\Gamma \Proves \phi \Implies(\Not \phi \Implies \psi)}}
{\Gamma, \phi \Proves\Not \phi \Implies \psi}}
{\Gamma, \phi, \Not \phi \Proves \psi.}
\end{aligned}
\end{equation*}

\quad

\begin{equation*}
\begin{aligned}
\drv{
\drv{}{
\Gamma, \phi, \Not \phi \Proves \Not \Not \phi}
&
\drv{}{
\Gamma, \phi, \Not \Not \phi \Proves \Not \Not \phi}}{
\Gamma, \phi \Proves \Not \Not \phi.}
\end{aligned}
\end{equation*}
\end{proof}

\begin{proposition}\label{2.4}
The following inference rules are derivable:
\begin{equation*}
\begin{aligned}
\drv
{\Gamma\Proves \phi}
{\Gamma \Proves \Not \Not \phi}
\end{aligned}\;, \qquad \qquad
\begin{aligned}
\drv
{\Gamma \Proves \Not \Not \phi}
{\Gamma \Proves \phi}
\end{aligned}\;,
\end{equation*}

\begin{equation*}
\begin{aligned}
\drv
{\Gamma, \phi \Proves \psi & \Gamma, \phi \Proves \Not \psi}
{\Gamma  \Proves \Not \phi}
\end{aligned}\;, \qquad \qquad
\begin{aligned}
\drv
{\Gamma, \Not \phi \Proves \psi & \Gamma, \Not \phi \Proves \Not \psi}
{\Gamma  \Proves \phi}
\end{aligned}\;,
\end{equation*}

\begin{equation*}
\begin{aligned}
\drv
{\Gamma, \phi \Proves \Not \phi}
{\Gamma  \Proves \Not \phi}
\end{aligned}\;, \qquad \qquad
\begin{aligned}
\drv
{\Gamma, \Not \phi \Proves \phi}
{\Gamma  \Proves \phi}
\end{aligned}\;.
\end{equation*}
\end{proposition}

\begin{proof}
Refer to Proposition~\ref{2.2} and Lemma~\ref{2.3}.

\begin{equation*}
\drv
{\Gamma \Proves \phi
&\drv
{}
{\Gamma, \phi \Proves \Not \Not \phi}}
{\Gamma \Proves \Not \Not \phi;}
\qquad \qquad
\drv
{\Gamma \Proves \Not \Not  \phi
&\drv
{}
{\Gamma, \Not \Not \phi \Proves \phi}}
{\Gamma \Proves \phi.}
\end{equation*}

\begin{equation*}
\drv
{
\drv
{\Gamma, \phi \Proves \psi & \Gamma, \phi \Proves \Not \psi}
{\Gamma, \phi \Proves \Not \phi}
&
\drv{}
{\Gamma, \Not \phi \Proves \Not \phi}
}
{\Gamma \Proves \Not \phi;}
\qquad \qquad
\drv
{
\drv
{\Gamma, \Not \phi \Proves \psi & \Gamma, \Not \phi \Proves \Not \psi}
{\Gamma \Proves \Not \Not \phi}
}
{\Gamma  \Proves \phi.}
\end{equation*}

\begin{equation*}
\drv
{\drv{}{\Gamma,  \phi  \Proves  \phi} & \Gamma ,\phi \Proves \Not \phi}
{\Gamma  \Proves \Not \phi;}
\qquad \qquad
\drv
{\Gamma,  \Not \phi  \Proves  \phi & \drv{}{\Gamma,\Not\phi \Proves \Not \phi}}
{\Gamma  \Proves \phi.}
\end{equation*}
\end{proof}

\begin{lemma}\label{2.5}
The following inference rules are derivable:
\begin{equation*}
\begin{aligned}
\drv
{\Gamma, \phi \Proves \psi}
{\Gamma, \phi, \phi \Proves \psi}
\end{aligned}\;, \qquad \qquad
\begin{aligned}
\drv
{\Gamma, \phi, \phi\Proves \psi}
{\Gamma, \phi\Proves \psi}
\end{aligned}\;,
\end{equation*}
\begin{equation*}
\begin{aligned}
\drv
{\Gamma, \phi\Proves \psi}
{\Gamma, \Not \Not \phi  \Proves \psi}
\end{aligned}\;, \qquad \qquad
\begin{aligned}
\drv
{\Gamma, \Not \Not \phi \Proves \psi}
{\Gamma, \phi \Proves \psi}
\end{aligned}\;.
\end{equation*}
\end{lemma}

\begin{proof}
Refer to Lemma~\ref{2.3}.

\begin{equation*}
\drv
{\drv{}{\Gamma, \phi \Proves \phi} & \Gamma, \phi\Proves \psi}
{\Gamma, \phi, \phi\Proves \psi;}
\qquad \qquad
\drv
{\drv{}{\Gamma, \phi \Proves \phi} & \Gamma, \phi, \phi \Proves \psi}
{\Gamma, \phi \Proves \psi.}
\end{equation*}

\begin{equation*}
\drv
{\drv
{}
{\Gamma, \Not \Not \phi \Proves \phi}
&\drv
{\drv
{}
{\Gamma, \phi, \Not \Not \phi \Proves \phi}
&\drv
{\drv
{}
{\Gamma, \phi \Proves \Not \Not \phi}
&\Gamma, \phi \Proves \psi}
{\Gamma, \phi, \Not \Not \phi \Proves \psi}
&\drv
{}
{\Gamma, \Not \Not \phi , \phi \Proves \Not \Not \phi}}
{\Gamma, \Not \Not \phi, \phi \Proves \psi}}
{\Gamma, \Not \Not \phi \Proves \psi;}
\end{equation*}
\begin{equation*}
\drv
{\drv
{}
{\Gamma, \phi \Proves \Not \Not \phi}
&\drv
{\drv
{}
{\Gamma,  \Not \Not \phi, \phi \Proves  \Not \Not \phi}
&\drv
{\drv
{}
{\Gamma,  \Not \Not \phi \Proves \phi}
&\Gamma,  \Not \Not \phi \Proves \psi}
{\Gamma,  \Not \Not  \phi,\phi \Proves \psi}
&\drv
{}
{\Gamma, \phi ,  \Not \Not \phi \Proves\phi}}
{\Gamma, \phi,  \Not \Not \phi \Proves \psi}}
{\Gamma, \phi \Proves \psi.}
\end{equation*}
\end{proof}

\begin{theorem}\label{2.6}
The inference rules in Figure \ref{F3} are derivable.
\end{theorem}

\begin{proof}
In the special case that $\Delta$ is empty, each of the rules in Figure \ref{F3} is derivable either because it is a primitive rule of $NOM$ or by Lemmas \ref{2.3} and \ref{2.5}. For each of these rules, the general case follows from this special case via the $\Implies$-introduction and $\Implies$-elimination rules, as in Lemma~\ref{2.3}. For illustration, we prove that the generalized paste rule is derivable. The proof proceeds by induction on the length of $\Delta$. The special case that $\Delta$ is empty is the base case of the induction.

Let $\ell\geq0$, and assume that the generalized paste rule is derivable for all sequences $\Delta$ of length $\ell$. Let $\Delta'$ be a sequence of length $\ell+1$; it is of the form $\Delta, \delta'$ for some sequence $\Delta$ of length $\ell$ and some formula $\delta'$. We now derive $\Gamma, \phi, \Delta, \delta' \Proves \psi$ from $\Gamma \Proves \phi$ and $\Gamma, \Delta, \delta' \Proves \psi$:
\begin{equation*}
\drv{
\drv{
\Gamma \Proves \phi &
\drv
{\Gamma, \Delta, \delta' \Proves \psi}
{\Gamma, \Delta \Proves \delta' \Implies \psi}}
{\Gamma, \phi, \Delta \Proves \delta' \Implies \psi}}
{\Gamma, \phi, \Delta, \delta' \Proves \psi.}
\end{equation*}
In other words, we have derived $\Gamma, \phi , \Delta' \Proves \psi$ from $\Gamma \Proves \phi$ and $\Gamma, \Delta' \Proves \psi$. Proceeding by induction on $\ell$, we conclude that the generalized paste rule is derivable for all sequences $\Delta$ of arbitrary length. The derivability of the other nine rules is proved similarly.
\end{proof}

\begin{figure}[h]

\begin{equation*}
\drv
{\Gamma, \phi, \phi, \Delta \Proves \psi}
{\Gamma, \phi, \Delta \Proves \psi}
\qquad \qquad
\drv
{\Gamma, \phi, \Delta \Proves \psi}
{\Gamma, \phi, \phi, \Delta \Proves \psi}
\end{equation*}

\begin{equation*}
\drv
{\Gamma \Proves \phi & \Gamma, \phi, \Delta \Proves \psi}
{\Gamma,\Delta \Proves \psi}
\qquad \qquad
\drv
{\Gamma \Proves \phi & \Gamma,\Delta \Proves \psi}
{\Gamma, \phi,\Delta \Proves \psi}
\end{equation*}

\begin{equation*}
\drv
{\Gamma, \phi, \psi \Proves \phi &\Gamma, \phi, \psi,\Delta \Proves \chi &   \Gamma, \psi, \phi \Proves \psi}
{\Gamma, \psi, \phi,\Delta \Proves \chi}
\end{equation*}

\begin{equation*}
\drv
{}
{\Gamma, \Not \phi, \phi, \Delta \Proves \psi}
\qquad\qquad
\drv
{}
{\Gamma, \phi, \Not \phi, \Delta \Proves \psi}
\end{equation*}

\begin{equation*}
\drv
{\Gamma, \Not \Not \phi, \Delta  \Proves \psi}
{\Gamma, \phi, \Delta \Proves \psi}
\qquad\qquad
\drv
{\Gamma, \phi, \Delta \Proves \psi}
{\Gamma, \Not \Not \phi, \Delta  \Proves \psi}
\end{equation*}

\begin{equation*}
\drv
{\Gamma, \phi,\Delta \Proves \psi & \Gamma, \Not \phi,\Delta \Proves \psi}
{\Gamma,\Delta \Proves \psi}
\end{equation*}
\caption{Some generalized rules of inference that are derivable in $NOM$.}\label{F3}
\end{figure}

\begin{lemma}\label{2.7}
The following sequents are derivable:
\begin{enumerate}
\item $\Gamma, \phi \And \psi, \Not \phi \Proves \phi \And \psi$ and $\Gamma, \phi \And \psi, \Not \psi \Proves \phi \And \psi$,
\item $\Gamma, \Not \phi, \phi \And \psi \Proves \Not \phi$ and $\Gamma, \Not \psi, \phi \And \psi \Proves \Not \psi$,
\item $\Gamma, \Not (\phi \And \psi), \phi \Proves \Not (\phi \And \psi)$ and $\Gamma, \Not (\phi \And \psi), \psi \Proves \Not (\phi \And \psi)$,
\item $\Gamma, \phi, \Not (\phi \And \psi) \Proves \phi$ and $\Gamma, \psi, \Not (\phi \And \psi) \Proves \psi$.
\end{enumerate}
\end{lemma}

\begin{proof}
We exhibit the derivations of four of these sequents because the derivations of the other four sequents are entirely similar. Refer to Proposition~\ref{2.4} and Theorem~\ref{2.6}.
\begin{equation*}
\drv
{\drv
{\drv
{}
{\Gamma, \phi \And \psi \Proves \phi \And \psi}}
{\Gamma, \phi \And \psi \Proves \phi}
&\drv
{}
{\Gamma, \phi \And \psi, \phi, \Not \phi \Proves \phi \And \psi }}
{\Gamma, \phi \And \psi, \Not \phi \Proves \phi \And \psi.}
\qquad
\drv
{\drv
{\drv
{}
{\Gamma, \Not \phi, \phi \And \psi \Proves \phi \And \psi}}
{\Gamma, \Not \phi, \phi \And \psi \Proves \phi }
&\drv
{}
{\Gamma, \Not \phi, \phi \And \psi \Proves \phi \And \psi}}
{\Gamma, \Not \phi, \phi \And \psi, \phi \Proves \phi \And \psi;}
\end{equation*}

\begin{equation*}
\drv
{\drv
{\drv
{}
{\Gamma, \Not \phi, \phi \And \psi \Proves \phi \And \psi}}
{\Gamma, \Not \phi, \phi \And \psi \Proves \phi}
&\drv
{\drv
{}
{\Gamma, \Not \phi, \phi, \phi \And \psi \Proves \phi}
&\drv
{}
{\Gamma, \Not \phi, \phi, \phi \And \psi \Proves \Not \phi}
&\drv
{}
{\Gamma, \Not \phi, \phi \And \psi, \phi \Proves \phi \And \psi}}
{\Gamma, \Not \phi, \phi \And \psi, \phi \Proves \Not \phi}}
{\Gamma, \Not \phi, \phi \And \psi \Proves \Not \phi.}
\end{equation*}

\begin{equation*}
\drv
{
\drv
{\drv
{}
{\Gamma, \Not(\phi \And \psi), \phi \And \psi, \phi \Proves \phi \And \psi}
&\drv
{}
{\Gamma, \Not(\phi \And \psi), \phi \And \psi, \phi \Proves \Not (\phi \And \psi)}
&\drv
{\drv
{}
{\Gamma, \Not (\phi \And \psi), \phi, \phi \And \psi \Proves \phi \And \psi}}
{\Gamma, \Not (\phi \And \psi), \phi, \phi \And \psi \Proves \phi}}
{\Gamma, \Not (\phi \And \psi), \phi, \phi \And \psi \Proves \Not (\phi \And \psi)}
}
{\Gamma, \Not (\phi \And \psi), \phi \Proves \Not (\phi \And \psi).}
\end{equation*}

\begin{equation*}
\drv
{\drv
{\drv
{}
{\Gamma, \phi, \phi \And \psi \Proves \phi \And \psi}}
{\Gamma, \phi, \phi \And \psi \Proves \phi}
&\drv
{}
{\Gamma, \phi, \phi \And \psi, \Not (\phi \And \psi) \Proves \phi}
&\drv
{\drv
{\drv
{}
{\Gamma, \phi \And \psi \Proves \phi \And \psi}}
{\Gamma, \phi \And \psi \Proves \phi}
&\drv
{}
{\Gamma, \phi \And \psi \Proves \phi \And \psi}}
{\Gamma, \phi \And \psi, \phi \Proves \phi \And \psi}}
{\Gamma, \phi \And \psi,\phi, \Not (\phi \And \psi) \Proves \phi;}
\end{equation*}

\begin{equation*}
\drv
{
\drv
{}
{\Gamma, \phi \And \psi,\phi, \Not (\phi \And \psi) \Proves \phi}
&\drv{
\drv{}{\Gamma, \Not(\phi \And \psi), \phi \Proves \Not (\phi \And \psi)}&
\drv{}{\Gamma, \Not(\phi \And \psi), \phi \Proves \phi}}
{\Gamma, \Not(\phi \And \psi), \phi, \Not (\phi \And \psi) \Proves \phi}}
{\Gamma, \phi, \Not(\phi \And \psi) \Proves \phi.}
\end{equation*}
\end{proof}

\begin{proposition}\label{2.8}
The following inference rules are derivable:
\begin{equation*}
\begin{aligned}
\drv
{\Gamma, \phi, \Not (\phi \And \psi), \Delta \Proves \chi}
{\Gamma, \Not (\phi \And \psi), \phi, \Delta \Proves \chi}
\end{aligned}\;,
\qquad \qquad
\begin{aligned}
\drv
{\Gamma, \Not (\phi \And \psi), \phi, \Delta \Proves \chi}
{\Gamma, \phi, \Not (\phi \And \psi), \Delta \Proves \chi}
\end{aligned}\;,
\end{equation*}
\begin{equation*}
\begin{aligned}
\drv
{\Gamma, \psi, \Not (\phi \And \psi), \Delta \Proves \chi}
{\Gamma, \Not (\phi \And \psi), \psi, \Delta \Proves \chi}
\end{aligned}\;,
\qquad \qquad
\begin{aligned}
\drv
{\Gamma, \Not (\phi \And \psi), \psi, \Delta \Proves \chi}
{\Gamma, \psi, \Not (\phi \And \psi), \Delta \Proves \chi}
\end{aligned}\;.
\end{equation*}
\end{proposition}

\begin{proof}
We exhibit the derivation of one of these rules because the derivations of the other three rules are entirely similar. Refer to Theorem~\ref{2.6} and Lemma~\ref{2.7}.
\begin{equation*}
\drv
{\drv
{}
{\Gamma, \phi, \Not (\phi \And \psi) \Proves \phi}
&\Gamma, \phi, \Not (\phi \And \psi), \Delta \Proves \chi
&\drv
{}
{\Gamma, \Not (\phi \And \psi), \phi \Proves \Not (\phi \And \psi)}}
{\Gamma, \Not (\phi \And \psi), \phi, \Delta \Proves \chi.}
\end{equation*}
\end{proof}

\begin{proposition}\label{2.9}
The following inference rules are derivable:
\begin{equation*}
\begin{aligned}
\drv
{\Gamma \Proves \Not \phi}
{\Gamma  \Proves \Not (\phi \And \psi)}
\end{aligned}\;, \qquad \qquad
\begin{aligned}
\drv
{\Gamma \Proves \phi}
{\Gamma \Proves \Not (\Not \phi \And \psi)}
\end{aligned}\;,
\end{equation*}
\begin{equation*}
\begin{aligned}
\drv
{\Gamma \Proves \Not \psi}
{\Gamma  \Proves \Not (\phi \And \psi)}
\end{aligned}\;, \qquad \qquad
\begin{aligned}
\drv
{\Gamma \Proves \psi}
{\Gamma \Proves \Not (\phi \And \Not \psi)}
\end{aligned}\;.
\end{equation*}
\end{proposition}

\begin{proof}
We exhibit the derivations of two of these rules because the derivations of the other two rules are entirely similar. Refer to Proposition~\ref{2.4} and Lemma~\ref{2.7}.

\begin{equation*}
\drv
{
\drv
{\drv
{}
{\Gamma, \phi \And \psi \Proves \phi \And \psi}}
{\Gamma, \phi \And \psi \Proves \phi}&
\drv
{
\Gamma \Proves \Not \phi &
\drv
{}
{\Gamma, \Not \phi,\phi \And \psi \Proves \Not \phi}}
{\Gamma, \phi \And \psi \Proves \Not \phi}}
{\Gamma  \Proves \Not (\phi \And \psi);}
\qquad \qquad
\drv
{\drv
{\Gamma \Proves \phi}
{\Gamma \Proves \Not \Not \phi}}
{\Gamma \Proves \Not (\Not \phi \And \psi).}
\end{equation*}
\end{proof}

\begin{theorem}\label{2.10}
The following inference rules are derivable:
\begin{equation*}
\begin{aligned}
\drv
{\Gamma \Proves \phi \Implies  \psi}
{\Gamma  \Proves \Not (\phi \And \Not(\phi \And \psi))}
\end{aligned}\;, \qquad \qquad
\begin{aligned}
\drv
{\Gamma  \Proves \Not (\phi \And \Not(\phi \And \psi))}
{\Gamma \Proves \phi \Implies  \psi}
\end{aligned}\;.
\end{equation*}
\end{theorem}

\begin{proof}
Refer to Proposition~\ref{2.9}.

\begin{equation*}
\drv
{\drv
{\drv
{\drv
{}
{\Gamma, \phi \Proves \phi}
& 
\drv
{\Gamma \Proves \phi \Implies \psi}
{\Gamma, \phi \Proves \psi}}
{\Gamma, \phi \Proves \phi \And \psi}}
{\Gamma, \phi \Proves \Not (\phi \And \Not (\phi \And \psi))} &
\drv
{\drv
{}
{\Gamma, \Not \phi \Proves \Not \phi}}
{\Gamma, \Not \phi \Proves \Not (\phi \And \Not (\phi \And \psi))}
}
{\Gamma  \Proves \Not (\phi \And \Not(\phi \And \psi)).}
\end{equation*}

\quad

Refer to Propositions \ref{2.4} and \ref{2.8}.

\begin{equation*}
\drv
{\Gamma \Proves \Not (\phi \And \Not(\phi \And \psi))
&\drv
{\drv
{\drv
{}
{\Gamma, \phi, \Not (\phi \And \psi), \Not(\phi \And \Not(\phi \And \psi)) \Proves \Not(\phi \And \Not(\phi \And \psi))}}
{\Gamma, \phi, \Not(\phi \And \Not(\phi \And \psi)), \Not (\phi \And \psi) \Proves \Not(\phi \And \Not(\phi \And \psi))}}
{\Gamma, \Not(\phi \And \Not(\phi \And \psi)),  \phi, \Not (\phi \And \psi) \Proves \Not(\phi \And \Not(\phi \And \psi))}}
{\Gamma, \phi, \Not (\phi \And \psi) \Proves \Not (\phi \And \Not(\phi \And \psi));}
\end{equation*}

\begin{equation*}
\drv
{\drv
{\drv
{\drv
{\drv
{\drv
{}
{\Gamma, \Not (\phi \And \psi), \phi \Proves \phi}}
{\Gamma, \phi, \Not (\phi \And \psi) \Proves \phi}
&\drv
{}
{\Gamma, \phi, \Not (\phi \And \psi) \Proves \Not(\phi \And \psi)}}
{\Gamma, \phi, \Not (\phi \And \psi) \Proves \phi \And \Not(\phi \And \psi)}
&\drv
{\Gamma \Proves \Not (\phi \And \Not(\phi \And \psi))}
{\Gamma, \phi, \Not (\phi \And \psi) \Proves \Not(\phi \And \Not(\phi \And \psi))}}
{\Gamma, \phi \Proves \phi \And \psi}}
{\Gamma, \phi \Proves \psi}}
{\Gamma \Proves \phi \Implies \psi.}
\end{equation*}
\end{proof}

\section{Noncommutativity, soundness, and completeness}\label{S3}

This section contains proofs establishing the basic metamathematical properties of $NOM$. We begin with the admissibility of the weakening rule.

\begin{proposition}\label{3.1}
If a sequent $\Gamma \Proves \phi$ is derivable, then $\Delta,\Gamma \Proves \phi$ is also derivable.
\end{proposition}

\begin{proof}
Recall that the length of a derivation is the number of inferences in that derivation. We prove the proposition by induction on the length of a derivation of $\Gamma \Proves \phi$. If $\Gamma \Proves \phi$ has a derivation of length one, then it is of the form $\Gamma_0, \phi \Proves \phi$, and hence $\Delta, \Gamma \Proves \phi$ is of the form $\Delta,\Gamma_0, \phi \Proves \phi$. Therefore, if $\Gamma \Proves \phi$ has a derivation of length $\ell = 1$, then $\Delta, \Gamma \Proves \phi$ is derivable.

Let $\ell > 1$, and assume that the proposition holds for all sequents that have a derivation of length smaller than $\ell$. Let $\Gamma \Proves \phi$ be a sequent that has a derivation of length $\ell$. Then, there exist sequents $\Gamma_1 \Proves \phi_1$, $\Gamma_2 \Proves \phi_2$, and $\Gamma_3 \Proves \phi_3$, which may be identical, that have derivations of length less than $\ell$ and that yield $\Gamma \Proves \phi$ via a primitive inference of $NOM$:
\begin{equation*}
\begin{aligned}
\drv
{\Gamma_1 \Proves \phi_1 & \Gamma_2 \Proves \phi_2 & \Gamma_3 \Proves \phi_3}
{\Gamma \Proves \phi}
\end{aligned}\;.
\end{equation*}
Perusing Figure~\ref{F1}, we conclude that
\begin{equation*}
\begin{aligned}
\drv
{\Delta, \Gamma_1 \Proves \phi_1 & \Delta, \Gamma_2 \Proves \phi_2 & \Delta, \Gamma_3 \Proves \phi_3}
{\Delta, \Gamma \Proves \phi}
\end{aligned}
\end{equation*}
is also a primitive inference of $NOM$. By the induction hypothesis, the sequents $\Delta, \Gamma_1 \Proves \phi_1$, $\Delta, \Gamma_2 \Proves \phi_2$, and $\Delta, \Gamma_3 \Proves \phi_3$ are all derivable, and thus the sequent $\Delta, \Gamma \Proves \phi$ is also derivable. Therefore, the proposition holds for all sequents that have a derivation of length smaller than or equal to $\ell$. By induction on $\ell$, the proposition holds for all sequents that have a derivation of any length.
\end{proof}

Let $H$ be the Hilbert-type deductive system given in \cite{Kleene}*{sec.~19}. Let $LK$ be the deductive system $G1$ given in \cite{Kleene}*{sec.~77}.
Let the \emph{exchange rule} (E) be the following rule of inference:
\begin{equation*}
\drv
{\Gamma, \phi, \psi, \Delta \Proves \chi}
{\Gamma, \psi, \phi, \Delta \Proves \chi.}
\end{equation*}

\begin{theorem}\label{3.2}
Let $\phi_1, \ldots, \phi_n \Proves \psi_0$ be a sequent. Then, the following are equivalent:
\begin{enumerate}
\item $\phi_1, \ldots, \phi_n \Proves \psi_0$ is derivable in $NOM + E$;
\item $\phi_1, \ldots, \phi_n \Proves \psi_0$ is derivable in $LK$;
\item $(\phi_1 \And \cdots \And \phi_n) \Implies \psi_0$ is derivable in $H$.
\end{enumerate}
When $n=0$, the expression $(\phi_1 \And \cdots \And \phi_n) \Implies \psi_0$ means $\psi_0$.
\end{theorem}

\begin{proof}
To prove the implication $(1) \Rightarrow (2)$, it is sufficient to prove that each of the inference rules of $NOM + E$ is derivable in $LK$. This is a set of exercises in introductory logic. We further remark only that each inference rule of $NOM+E$ is derivable in $LK$ because $LK$ does include the cut rule, which is merely admissible rather than derivable in the rest of that system. The implication $(2) \Rightarrow (3)$ is essentially just \cite{Kleene}*{Thm.~47}\cite{Gentzen2}*{sec.~V.6}. Indeed the deduction theorem for $H$ implies that $\psi_0$ may be derived from the assumptions $\phi_1, \ldots, \phi_n$ if and only if the formula $(\phi_1 \And \cdots \And \phi_n) \Implies \psi_0$ may be derived without any assumptions \cite{Kleene}*{Thm.~1}. It remains to prove the implication $(3) \Rightarrow (1)$.

Assume that  the formula $(\phi_1 \And \cdots \And \phi_n) \Implies \psi_0$ has a derivation in $H$. As a consequence of the cut elimination theorem, this formula has a derivation in $H$ that uses only those logical axioms of $H$ that contain the symbols $\And$, $\Implies$, and $\Not$ \cite{Kleene}*{Thm.~46, Cor.~2, Thm.~47}. For each such logical axiom $\alpha$, the sequent $\Proves \alpha$ is derivable in $NOM + E$: The sequent $\Proves\phi \Implies (\psi \Implies \phi)$ has the derivation
\begin{equation*}
\drv
{\drv
{\drv
{\drv
{}
{\psi, \phi \Proves \phi}}
{\phi, \psi \Proves \phi}}
{\phi \Proves \psi \Implies \phi}}
{\Proves\phi \Implies (\psi \Implies \phi).}
\end{equation*}
The sequent $\Proves (\phi \Implies \psi) \Implies ((\phi \Implies (\psi \Implies \chi)) \Implies (\phi \Implies \chi))$ has the derivation
\begin{equation*}
\drv
{\drv
{\drv
{\drv
{\drv
{\drv
{\drv
{}
{\phi \Implies (\psi \Implies \chi),\phi \Implies \psi \Proves \phi \Implies \psi}}
{\phi \Implies (\psi \Implies \chi),\phi \Implies \psi, \phi \Proves \psi}}
{\phi \Implies \psi,\phi \Implies (\psi \Implies \chi), \phi \Proves \psi}
&\drv
{\drv
{\drv
{}
{\phi \Implies \psi,\phi \Implies (\psi \Implies \chi) \Proves \phi \Implies (\psi \Implies \chi)}}
{\phi \Implies \psi,\phi \Implies (\psi \Implies \chi), \phi \Proves \psi \Implies \chi}}
{\phi \Implies \psi,\phi \Implies (\psi \Implies \chi), \phi, \psi \Proves \chi}}
{\phi \Implies \psi,\phi \Implies (\psi \Implies \chi), \phi \Proves \chi}}
{\phi \Implies \psi,\phi \Implies (\psi \Implies \chi) \Proves \phi \Implies \chi}}
{\phi \Implies \psi \Proves(\phi \Implies (\psi \Implies \chi)) \Implies (\phi \Implies \chi)}}
{\Proves (\phi \Implies \psi) \Implies ((\phi \Implies (\psi \Implies \chi)) \Implies (\phi \Implies \chi)).}
\end{equation*}
The sequent $\Proves \phi \Implies (\psi \Implies (\phi \And \psi))$ has the derivation
\begin{equation*}
\drv
{\drv
{\drv
{\drv
{\drv
{}
{\psi, \phi \Proves \phi}}
{\phi, \psi \Proves \phi}
&\drv
{}
{\phi, \psi \Proves \psi}}
{\phi, \psi \Proves \phi \And \psi}}
{\phi \Proves \psi \Implies (\phi \And \psi)}}
{\Proves \phi \Implies (\psi \Implies (\phi \And \psi)).}
\end{equation*}
The sequents $\Proves (\phi \And \psi) \Implies \phi$ and $\Proves (\phi \And \psi) \Implies \psi$ have the derivations
\begin{equation*}
\drv
{\drv
{\drv
{}
{\phi \And \psi \Proves \phi \And \psi}}
{\phi \And \psi \Proves \phi}}
{\Proves (\phi \And \psi) \Implies \phi,}
\qquad\qquad
\drv
{\drv
{\drv
{}
{\phi \And \psi \Proves \phi \And \psi}}
{\phi \And \psi \Proves \psi}}
{\Proves (\phi \And \psi) \Implies \psi.}
\end{equation*}
By Lemma~\ref{2.4}, the sequent $\Proves (\phi \Implies \psi) \Implies ((\phi \Implies \Not \psi) \Implies \Not \phi)$ has the derivation
\begin{equation*}
\drv
{\drv
{\drv
{\drv
{\drv
{\drv
{}
{\phi \Implies \Not \psi, \phi \Implies \psi \Proves \phi \Implies \psi}}
{\phi \Implies \Not \psi, \phi \Implies \psi, \phi \Proves \psi}}
{\phi \Implies \psi, \phi \Implies \Not \psi, \phi \Proves \psi}
&\drv
{\drv
{}
{\phi \Implies \psi, \phi \Implies \Not \psi \Proves \phi \Implies \Not \psi}}
{\phi \Implies \psi, \phi \Implies \Not \psi, \phi \Proves \Not \psi}}
{\phi \Implies \psi, \phi \Implies \Not \psi \Proves \Not \phi}}
{\phi \Implies \psi \Proves (\phi \Implies \Not \psi) \Implies \Not \phi }}
{\Proves (\phi \Implies \psi) \Implies ((\phi \Implies \Not \psi) \Implies \Not \phi).}
\end{equation*}
By Lemma~\ref{2.3}, the sequent $\Proves \Not \Not \phi \Implies \phi$ has the derivation
\begin{equation*}
\drv
{\drv
{}
{\Not \Not \phi \Proves \phi}}
{\Proves  \Not \Not \phi \Implies \phi.}
\end{equation*}
Therefore, for every logical axiom $\alpha$ in the derivation of $(\phi_1 \And \cdots \And \phi_n) \Implies \psi_0$, the sequent $\Proves \alpha$ is derivable in $NOM+E$.

Furthermore, by Proposition~\ref{2.1}, the inference rule
\begin{equation*}
\drv{\Proves \phi & \Proves \phi \Implies \psi}{\Proves \psi}
\end{equation*}
is also derivable in $NOM+E$. Therefore, we may transform a derivation of $(\phi_1 \And \cdots \And \phi_n) \Implies \psi_0$ in $H$ into a derivation of $\Proves (\phi_1 \And \cdots \And \phi_n) \Implies \psi_0$ in $NOM+E$ essentially by placing the symbol $\Proves$ in front of each formula. Formally, this is a straightforward inductive argument. We conclude that the sequent $\Proves (\phi_1 \And \cdots \And \phi_n) \Implies \psi_0$ is derivable in $NOM+E$.

By Proposition \ref{3.1} and the $\Implies$-elimination rule, the sequent $\phi_1, \ldots, \phi_n, \phi_1 \And \cdots \And \phi_n \Proves \psi_0$ is also derivable in $NOM+E$. Certainly, so is the sequent $\phi_1, \ldots, \phi_n \Proves \phi_1 \And \cdots \And \phi_n$. Appealing to the cut rule, we conclude that $\phi_1, \dots, \phi_n \Proves \phi_0$ is derivable in $NOM+E$, proving the implication $(3)  \Rightarrow (1)$.
\end{proof}

In any orthomodular lattice $\Q$, the binary operations $\Andthen$ and $\Implies$ are defined by $a \Andthen b = (a \Or \Not b) \And b$ and $a \Implies b = \Not a \Or (a \And b)$ for all $a, b \in \Q$. Our convention is that $\Andthen$ associates to the left and $\Implies$ associates to the right. Furthermore, $a_1 \Andthen \cdots \Andthen a_n$ denotes $\top$ when $n=0$.

\begin{definition}\label{3.3}
Let $\Phi$ be the set of formulas of $NOM$. We define an \emph{interpretation} to be a function $\[\cdot\]\: \Phi \to \Q$ such that $\Q$ is an orthomodular lattice. We define a sequent $\phi_1, \ldots, \phi_n \Proves \psi$ to be \emph{true} in such an interpretation if $\[\phi_1\]\Andthen \cdots \Andthen \[\phi_n\] \leq \[\psi\]$. We define a rule of inference to be \emph{sound} in such an interpretation if its conclusion is true whenever its premises are all true.
\end{definition} 

\begin{proposition}\label{3.4}
Let $\[\cdot\]$ be a surjective interpretation. If the rules of inference of $NOM$ are all sound in this interpretation, then all formulas $\phi$ and $\psi$ satisfy
\begin{enumerate}
\item $\[\phi \And \psi\] = \[\phi\] \And \[\psi\]$,
\item $\[\phi \Implies \psi\] = \[\phi\] \Implies \[\psi\]$,
\item $\[\Not\phi\] = \Not \[\phi\]$.
\end{enumerate}
\end{proposition}

\begin{proof}
Let $\Q$ be the orthomodular lattice that is the codomain of the interpretation $\[\cdot\]$. Assume that the rules of inference of $NOM$ are all sound.

Let $\phi$ be a formula. Let $\alpha$ be a formula such that $\[\alpha\] = \bot$, and let $\beta$ be a formula such that $\[\beta\] = \[\phi\] \Or \[\Not \phi\]$. By assumption, the following derivations both yield sound rules of inference:
\begin{equation*}
\drv
{\drv
{}
{\phi \Proves \phi}}
{\phi, \Not \phi \Proves \alpha,}
\qquad \qquad 
\drv
{\phi \Proves \beta & \Not \phi \Proves \beta}
{\Proves \beta.}
\end{equation*}
The soundness of the first rule implies that $\[\phi\] \Andthen \[\Not\phi\]  \leq \[\alpha\] = \bot$, and thus, $\[\phi\] \perp \[\Not\phi\]$. We certainly have that $\[\phi\] \leq \[\beta\]$ and $\[\Not \phi\] \leq \[\beta\]$, so the soundness of the second rule implies that $\top \leq \[\beta\] = \[\phi\] \Or \[\Not \phi\]$. Altogether, we have that $\[\Not \phi\]$ is a complement of $\[\phi\]$ that is orthogonal to $\[\phi\]$. Therefore, $\[\Not \phi\] = \Not \[\phi\]$ for all formulas $\phi$.

Let $\phi$ and $\psi$ be formulas. Let $\gamma$ be a formula such that $\[\gamma\] = \[\phi\] \And \[\psi\]$.
By assumption, the following derivations all yield sound rules of inference:
\begin{equation*}
\drv
{\drv
{}
{\phi \And \psi \Proves \phi \And \psi}}
{\phi \And \psi \Proves \phi,}
\qquad\qquad
\drv
{\drv
{}
{\phi \And \psi \Proves \phi \And \psi}}
{\phi \And \psi \Proves \psi,}
\qquad \qquad
\drv
{\gamma \Proves \phi & \gamma \Proves \psi}
{\gamma \Proves \phi \And \psi.}
\end{equation*}
The soundness of the first two rules implies that $\[\phi \And \psi\] \leq \[\phi\]$ and $\[\phi \And \psi\] \leq \[\psi\]$; thus, $\[\phi \And \psi\]$ is a lower bound of $\[\phi\]$ and $\[\psi\]$. We certainly have that $\[\gamma\] \leq \[\phi\]$ and $\[\gamma\] \leq \[\psi\]$, so the soundness of the third rule implies that $\[\phi\] \And \[\psi\] = \[\gamma\] \leq \[\phi \And\psi\]$. Altogether, we have that $\[\phi \And \psi\]$ is a lower bound of $\[\phi\]$ and $\[\psi\]$ that is at least as large as their greatest lower bound $\[\phi\] \And \[\psi\]$. Therefore, $\[\phi \And\psi\] = \[\phi\] \And \[\psi\]$ for all formulas $\phi$ and $\psi$.

Let $\phi$ and $\psi$ be formulas. By Theorem~\ref{2.10}, the following derivations yield sound rules of inference:
\begin{equation*}
\drv
{\drv
{}
{\phi \Implies \psi \Proves \phi \Implies \psi}}
{\phi \Implies \psi \Proves \Not (\phi \And \Not (\phi \And \psi)),}
\qquad \qquad
\drv
{\drv
{}
{\Not (\phi \And \Not (\phi \And \psi)) \Proves \Not (\phi \And \Not (\phi \And \psi))}}
{\Not (\phi \And \Not (\phi \And \psi)) \Proves \phi \Implies \psi.}
\end{equation*}
The soundness of the first rule implies that $\[ \phi \Implies \psi\] \leq \[\Not (\phi \And \Not (\phi \And \psi))\]$, and the soundness of the second rule implies the opposite inequality. Altogether, we have that
\begin{align*}&
\[ \phi \Implies \psi\]
=
\[\Not (\phi \And \Not (\phi \And \psi))\]
=
\Not \[\phi \And \Not (\phi \And \psi)\]
=
\Not (\[ \phi\] \And \[\Not (\phi \And \psi)\])
\\ & =
\Not (\[ \phi \] \And \Not \[ \phi \And \psi\])
=
\Not (\[\phi\] \And \Not (\[\phi\] \And \[\psi\]))
=
\Not \[\phi\] \Or (\[\phi\] \And \[\psi\])
= 
\[\phi\] \Implies \[\psi\].
\end{align*}
Therefore, $\[\phi \Implies \psi\] = \[\phi\] \Implies \[\psi\]$ for all formulas $\phi$ and $\psi$.
\end{proof}

\begin{theorem}\label{3.5}
Let $\[\cdot\]$  be an interpretation. If all formulas $\phi$ and $\psi$ satisfy
\begin{enumerate}
\item $\[\phi \And \psi\] = \[\phi\] \And \[\psi\]$,
\item $\[\phi \Implies \psi\] = \[\phi\] \Implies \[\psi\]$,
\item $\[\Not\phi\] = \Not \[\phi\]$,
\end{enumerate}
then the rules of inference of $NOM$ are all sound in this interpretation.
\end{theorem}

\begin{proof}
For brevity, we use the notation $\[\Gamma\] = \[\gamma_1\] \Andthen \cdots \Andthen \[\gamma_n\]$, where $\Gamma$ is the sequence $\gamma_1, \ldots, \gamma_n$. 

First, we prove the soundness of the conjunction and implication rules. Let $\Gamma \Proves \phi$ and $\Gamma \Proves \psi$ be sequents, and assume that they are true. Then, we have that $\[\Gamma\] \leq \[\phi\]$ and $\[\Gamma\] \leq \[\psi\]$, and hence $\[\Gamma\] \leq \[\phi\] \And \[\psi\] = \[\phi \And \psi\]$. In other words, $\Gamma \Proves \phi \And \psi$ is true. Therefore, the $\And$-introduction rule is sound. We may similarly conclude that the $\And$-eliminiation rules are sound.

The $\Implies$-introduction and $\Implies$-elimination rules are both sound if and only if the inequality $\[\Gamma\] \Andthen \[\phi\] \leq \[\psi\]$ is equivalent to the inequality $\[ \Gamma \]\leq \[\phi\] \Implies \[\psi\]$ for all $\Gamma$, $\phi$, and $\psi$. These inequalities certainly are equivalent because the inequalities $a \Andthen b \leq c$ and $a \leq b \Implies c$ are equivalent in every orthomodular lattice \cite{Finch}. Therefore, the $\Implies$-introduction rule and the $\Implies$-elimination rule are sound.

Of the six remaining rules of inference, we may prove the soundness of all but the conditional exchange rule by appealing to the soundness of the rules for conjunction and implication and to the theorem that the orthomodular lattice $2^4 \times MO2$ is the free orthomodular lattice on two generators \cite{Kalmbach}*{Thm.~2.1}. This theorem implies that an inequality in two variables holds universally in all orthomodular lattices if it holds in $2$ and in $MO2$, making the verification of such inequalities a matter of routine computation.

For illustration, we prove the soudness of the excluded middle rule. Let $\Gamma, \phi \Proves \psi$ and $\Gamma, \Not \phi \Proves \psi$ be sequents, and assume that they are true. By the soundness of the conjunction and implication rules, we find that $\Gamma \Proves (\phi \Implies \psi ) \And (\Not \phi \Implies \psi)$ is also true. By routine computation, we find that the inequality $(a \Implies b) \And (\Not a \Implies b) \leq b$ holds universally in any orthomodular lattice. We infer that \begin{equation*}\[\Gamma\] \leq \[(\phi \Implies \psi ) \And (\Not \phi \Implies \psi)\] = (\[\phi\] \Implies \[\psi \]) \And (\Not \[\phi\] \Implies \[\psi\]) \leq \[\psi\].\end{equation*}
Thus, the sequent $\Gamma \Proves \psi$ is true. Therefore, the excluded middle rule is sound. We may similarly conclude that the deductive explosion rule, the assumption rule, the cut rule, and the paste rule are all sound.

That leaves the conditional exchange rule, which is not directly amenable to the same approach. However, we can use the same approach to prove the soundness of the rule
\begin{equation*}
\begin{aligned}
\drv
{\Gamma, \phi, \psi \Proves \phi & \Gamma, \psi, \phi \Proves \psi}
{\Gamma, \Not \phi, \psi \Proves \Not \phi}
\end{aligned}\;.
\end{equation*}
We can then prove the soundness of the conditional exchange rule by deriving it from the rules whose soundness we have already established:
\begin{equation*}
\drv
{\drv
{\drv
{\drv
{\drv
{\drv
{\Gamma, \phi, \psi \Proves \phi
&\Gamma, \phi, \psi \Proves \chi}
{\Gamma, \phi, \psi, \phi \Proves \chi}}
{\Gamma, \phi, \psi \Proves \phi \Implies \chi}}
{\Gamma, \phi \Proves \psi \Implies(\phi \Implies \chi)}
&\drv
{\drv
{\drv
{\drv
{\Gamma, \phi, \psi \Proves \phi 
&\Gamma, \psi, \phi \Proves \psi}
{\Gamma, \Not \phi, \psi \Proves \Not \phi}}
{\Gamma, \Not \phi, \psi, \phi \Proves \chi}}
{\Gamma, \Not \phi, \psi \Proves \phi \Implies \chi}}
{\Gamma, \Not \phi \Proves \psi \Implies(\phi \Implies \chi) }}
{\Gamma \Proves \psi \Implies(\phi \Implies \chi)}}
{\Gamma, \psi \Proves \phi \Implies \chi}}
{\Gamma, \psi, \phi \Proves \chi.}
\end{equation*}
We have now established the soundness of all eleven inference rules of $NOM$.
\end{proof}

\begin{lemma}\label{3.6}
The deductive system $NOM$ has the following properties:
\begin{enumerate}
\item if $\phi \Proves \psi$ and $\psi \Proves \chi$ are both derivable, then $\phi \Proves \chi$ is derivable;
\item if $\phi \Proves \psi$ is derivable, then $\Not \psi \Proves \Not \phi$ is derivable.
\end{enumerate}
\end{lemma}

\begin{proof}
Assume that $\phi \Proves \psi$ and $\psi \Proves \chi$ are both derivable. By Proposition~\ref{3.1}, the sequent $\phi, \psi \Proves \chi$ is also derivable, and therefore, so is the sequent $\phi \Proves \chi$:
\begin{equation*}
\drv
{\drv
{}
{\phi \Proves \psi}
&\drv
{}
{\phi, \psi \Proves \chi}}
{\phi \Proves \chi.}
\end{equation*}

Assume that $\phi \Proves \psi$ is derivable. By Proposition~\ref{3.1}, the sequent $\Not\psi, \phi \Proves \psi$ is also derivable, and therefore, so is the sequent $\Not \psi \Proves \Not \phi$:
\begin{equation*}
\drv
{\drv
{}
{\Not \psi, \phi \Proves \psi}
&\drv
{\drv
{\drv
{}
{\Not \psi, \psi, \phi \Proves \psi}
&\drv
{}
{\Not \psi, \psi, \phi \Proves \Not \psi}
&\drv
{\drv
{}
{\Not \psi, \phi \Proves \psi}
&\drv
{}
{\Not \psi, \phi \Proves \phi}}
{\Not \psi, \phi, \psi \Proves \phi}}
{\Not \psi, \phi, \psi \Proves \Not \psi}}
{\Not \psi, \phi \Proves \Not \psi}}
{\Not \psi \Proves \Not \phi.}
\end{equation*}
Refer to Proposition~\ref{2.4} and Theorem~\ref{2.6}.
\end{proof}

\begin{definition}\label{3.7}
For all formulas $\phi$ and $\psi$ of $NOM$, we define $\phi \leq \psi$ if $\phi \Proves \psi$ is derivable in $NOM$. This relation is reflexive by Definition~\ref{1.1} and transitive by Lemma~\ref{3.6}(1). In other words, $(\Phi, \leq)$ is a preordered set.

We define $(\L,\leq)$ to be the partially ordered set that is obtained by identifying formulas that are equivalent in $(\Phi, \leq)$. Let $[\,\cdot\,]\: \Phi \to \L$ be the quotient map. The partially ordered set $(\L,\leq)$ has a canonical order-reversing involution $\Not$ that is defined by $\Not [\phi] = [\Not \phi]$. This operation is well-defined and order-reversing by Lemma~\ref{3.6}(2), and it is an involution by Lemma~\ref{2.3}.
\end{definition}

\begin{theorem}\label{3.8}
The structure $(\L, \leq, \Not)$ is an orthomodular lattice. Furthermore,
\begin{enumerate}
\item $[\phi \And \psi] = [\phi] \And [\psi]$,
\item $[\phi \Implies \psi] = [\phi] \Implies [\psi]$,
\item $[\Not \phi] = \Not [\phi]$,
\end{enumerate}
for all formulas $\phi$ and $\psi$.
\end{theorem}

\begin{proof}
Claim (1) is an immediate consequence of Definition~\ref{1.1}. Claim (3) is an immediate consequence of Definition~\ref{3.7}. With claims (1) and (3) in hand, claim (2) is equivalent to the equality $[\phi \Implies \psi] = [\Not (\phi \And \Not (\phi \And \psi))]$, which is an immediate consequence of Theorem~\ref{2.10}. Thus, claims (1), (2), and (3) are established.

Hence, $\L$ has binary meets that are given by $[\phi] \And [\psi] = [\phi \And \psi]$. Since $\L$ has an order-reversing involution, it also has binary joins, which are given by $[\phi] \Or [\psi] = \Not(\Not [\phi] \And \Not [\psi]) = [\Not (\Not \phi \And \Not \psi)]$.
For all formulas $\phi$ and $\psi$, the sequent $\phi \And \Not \phi \Proves \psi$ is derivable:
\begin{equation*}
\drv
{\drv
{\drv
{}
{\phi \And \Not \phi \Proves \phi \And \Not \phi}}
{\phi \And \Not \phi \Proves \phi}
&\drv
{\drv
{}
{\phi \And \Not \phi \Proves \phi \And \Not \phi}}
{\phi \And \Not \phi \Proves \Not \phi}}
{\phi \And \Not \phi \Proves \psi;}
\end{equation*}
we have appealed to Proposition~\ref{2.2}. Thus, $\L$ has a least element, and moreover, $[\phi] \And \Not [\phi]$ is equal to that least element for every formula $\phi$. Applying the order-reversing involution $\Not$, we find that $\L$ also has a greatest element, and moroever, $\Not[\phi] \Or [\phi]$ is equal to that greatest element for every formula $\phi$. Therefore, $(\L, \leq, \Not)$ is an ortholattice.

Let $\phi$ and $\psi$ be formulas such that $[\phi] \leq [\psi]$, that is, such that the sequent $\phi \Proves \psi$ is derivable.
By Proposition~\ref{3.1}, the sequent $\Not \phi, \psi, \phi \Proves \psi$ is also derivable. We apply Proposition~\ref{2.4} and Theorems \ref{2.6} and \ref{2.10} to derive $\psi \Proves \Not(\Not \phi \And \Not (\Not \phi \And \psi))$:
\begin{equation*}
\drv
{\drv
{\drv
{\drv
{\drv
{}
{\phi \Proves \psi}
&\drv
{}
{\phi, \Not \phi \Proves \psi}}
{\phi, \psi, \Not \phi \Proves \psi}
&\drv
{\drv
{\drv
{\drv
{}
{\Not \phi, \phi, \psi \Proves \phi}
&\drv
{}
{\Not \phi, \phi, \psi \Proves \Not \phi}
&\drv
{}
{\Not \phi, \psi, \phi \Proves \psi}}
{\Not \phi, \psi, \phi \Proves \Not \phi}}
{\Not \phi, \psi \Proves \Not \phi}
&\drv
{}
{\Not \phi, \psi \Proves \psi}}
{\Not \phi,\psi, \Not \phi \Proves \psi}}
{\psi, \Not \phi \Proves \psi}}
{\psi \Proves \Not \phi \Implies \psi}}
{\psi \Proves \Not(\Not \phi \And \Not (\Not \phi \And \psi)).}
\end{equation*}
Therefore, $[\psi] \leq [\Not(\Not \phi \And \Not (\Not \phi \And \psi))] = \Not(\Not [\phi] \And \Not (\Not [\phi] \And [\psi])) = [\phi] \vee (\Not [\phi] \And [\psi])$.

For the opposite inequality, we reason that $\Not \phi \Implies \psi, \phi \Proves \psi$ is derivable by Proposition~\ref{3.1}, and hence the sequent $\Not \phi \Implies \psi \Proves \psi$ is derivable:
\begin{equation*}
\drv
{\drv
{}
{\Not \phi \Implies \psi, \phi \Proves \psi}
&\drv
{\drv
{}
{\Not \phi \Implies \psi \Proves \Not \phi \Implies \psi}}
{\Not \phi \Implies \psi, \Not \phi \Proves \psi}}
{\Not \phi \Implies \psi \Proves \psi.}
\end{equation*}
Thus, $[\Not \phi \Implies \psi] \leq [\psi]$. By Theorem~\ref{2.10}, we also have that $[\Not(\Not \phi \And \Not (\Not \phi \And \psi))] \leq [\Not \phi \Implies \psi]$, and therefore, \begin{equation*}[\phi] \vee (\Not [\phi] \And [\psi]) = \Not(\Not [\phi] \And \Not (\Not [\phi] \And [\psi])) = [\Not(\Not \phi \And \Not (\Not \phi \And \psi))] \leq [\Not \phi \Implies \psi] \leq [\psi].\end{equation*} Altogether, we have that $[\phi] \vee (\Not [\phi] \And [\psi])=  [\psi]$ for all formulas $\phi$ and $\psi$ such that $[\phi] \leq [\psi]$.
We conclude that $(\L, \leq, \neg)$ is an orthomodular lattice, as claimed.
\end{proof}

\begin{corollary}\label{3.9}
If a sequent is true in every interpretation $\[\cdot\]$ such that
\begin{enumerate}
\item $\[\phi \And \psi\] = \[\phi\] \And \[\psi\]$,
\item $\[\phi \Implies \psi\] = \[\phi\] \Implies \[\psi\]$,
\item $\[\Not\phi\] = \Not \[\phi\]$,
\end{enumerate}
\nopagebreak for all formulas $\phi$ and $\psi$, then that sequent is derivable in $NOM$.
\end{corollary}

\begin{proof}
Let $\phi_1, \ldots, \phi_n \Proves \psi_0$ be a sequent that it is true in every interpretation that satisfies conditions \mbox{(1)\ndash(3)}. By Theorem~\ref{3.8}, the quotient map $[\,\cdot\,]\:\Phi \to \L$ is an interpretation that satisfies conditions (1)\ndash(3); hence $[\phi_1] \Andthen \cdots \Andthen [\phi_n] \leq [\psi_0]$. Repeatedly applying the adjunction between $\Andthen$ and $\Implies$ \cite{Finch}, we conclude that $[\phi_1] \Implies \cdots \Implies [\phi_n] \Implies [\psi_0]$ is the greatest element of $\L$.
By Theorem~\ref{3.8}(2), $[\phi_1 \Implies \cdots \Implies \phi_n \Implies \psi_0]$ is the greatest element of $\L$. In particular, the sequents $\psi_0 \Proves \phi_1 \Implies \cdots \Implies \phi_n \Implies \psi_0$ and $\Not \psi_0 \Proves \phi_1 \Implies \cdots \Implies \phi_n \Implies \psi_0$ are both derivable. Applying the excluded middle rule, we infer that $\Proves \phi_1 \Implies \cdots \Implies \phi_n \Implies \psi_0$ is derivable. Repeatedly applying the $\Implies$-elimination rule, we conclude that $\phi_1, \ldots, \phi_n \Proves \psi_0$ is derivable.
\end{proof}

\section{Disjunction and compatibility}\label{S4}

\begin{definition}\label{4.1}
We introduce two abbreviations:
\begin{enumerate}
\item let $\phi \Or \psi$ be an abbreviation for $\Not (\Not \phi \And \Not \psi)$;
\item let $\phi \Comp \psi$ be an abbreviation for $(\phi \Implies (\psi \Implies \phi)) \And (\psi \Implies (\phi \Implies \psi))$.
\end{enumerate}
\end{definition}

In this section, we derive the primitive inference rules for these two connectives. We also derive the equivalence between the given expression for $\phi \Comp \psi$ and the standard expression for it. Along the way, we derive a number of other rules of inference that express intelligible logical principles. We do not appeal to the results of section \ref{S3} in any way. This section exhibits the derivations whose existence it claims, continuing the development that was initiated in section \ref{S2}.

\begin{proposition}\label{4.2}
The following inference rule is derivable:
\begin{equation*}
\begin{aligned}
\drv
{\Gamma, \phi, \psi \Proves \Not \phi}
{\Gamma, \phi \Proves \Not \psi}
\end{aligned}\;.
\end{equation*}
\end{proposition}

\begin{proof}
Refer to Proposition~\ref{2.4} and Theorem~\ref{2.6}.
\begin{equation*}
\drv
{\drv
{\drv
{\drv
{}
{\Gamma,  \phi, \Not \phi, \psi \Proves \Not \phi}
&\drv
{}
{\Gamma,  \phi, \Not \phi, \psi \Proves  \phi}
&\drv
{\Gamma,  \phi, \psi \Proves \Not \phi
&\drv
{}
{\Gamma,  \phi, \psi \Proves \psi}}
{\Gamma,  \phi, \psi, \Not \phi \Proves \psi}}
{\Gamma,  \phi, \psi, \Not \phi \Proves  \phi}}
{\Gamma,  \phi, \psi \Proves  \phi}
& \Gamma,  \phi, \psi \Proves \Not \phi}
{\Gamma,  \phi \Proves  \Not \psi.}
\end{equation*}
\end{proof}

\begin{lemma}\label{4.3}
The following inference rule is derivable:
\begin{equation*}
\begin{aligned}
\drv
{\Gamma \Proves \Not (\phi \And \psi) & \Gamma, \chi, \Not \phi \Proves \Not \chi & \Gamma, \chi, \Not \psi \Proves \Not \chi}
{\Gamma \Proves \Not \chi}
\end{aligned}\;.
\end{equation*}
\end{lemma}

\begin{proof}
Refer to Propositions \ref{2.4} and \ref{4.2} and Theorem~\ref{2.6}.
\begin{equation*}
\drv
{\Gamma \Proves \Not (\phi \And \psi)
&\drv
{\drv
{\drv
{\Gamma \Proves \Not (\phi \And \psi)
& \drv
{\drv
{\drv
{\Gamma, \chi, \Not \phi \Proves \Not \chi}
{\Gamma, \chi, \Proves \Not \Not \phi}}
{\Gamma, \chi \Proves \phi}
&\drv
{\drv
{\Gamma, \chi, \Not \psi \Proves \Not \chi}
{\Gamma, \chi \Proves \Not \Not \psi}}
{\Gamma, \chi \Proves \psi}}
{\Gamma, \chi \Proves \phi \And \psi}}
{\Gamma, \Not ( \phi \And \psi), \chi \Proves \phi \And \psi}}
{\Gamma, \Not ( \phi \And \psi), \chi \Proves \Not \Not(\phi \And \psi)}}
{\Gamma, \Not (\phi \And \psi) \Proves \Not \chi}}
{\Gamma \Proves \Not \chi.}
\end{equation*}
\end{proof}

\begin{theorem}\label{4.4}
The following inference rule is derivable:
\begin{equation*}
\begin{aligned}
\drv
{\Gamma \Proves \psi & \Gamma, \phi \Proves \psi}
{\Gamma, \Not \phi \Proves \psi}
\end{aligned}\;.
\end{equation*}
\end{theorem}

\begin{proof}
Refer to Theorems \ref{2.6} and \ref{2.10}, Propositions \ref{2.8} and \ref{2.9}, and Lemma~\ref{4.3}.

\begin{equation*}
\drv
{\drv
{}
{\Gamma, \Not (\phi \And \psi) \Proves \Not (\phi \And \psi)}
&\drv
{}
{\Gamma,\Not (\phi \And \psi), \phi, \Not \phi \Proves \Not \phi}
&
\drv
{\drv
{\Gamma, \phi \Proves \psi
&\drv
{\drv
{}
{\Gamma, \phi, \Not (\phi \And \psi), \psi, \Not \psi \Proves \Not \phi}}
{\Gamma, \phi, \psi, \Not (\phi \And \psi), \Not \psi \Proves \Not \phi}}
{\Gamma, \phi,\Not (\phi \And \psi), \Not \psi \Proves \Not \phi}}
{\Gamma,\Not (\phi \And \psi), \phi, \Not \psi \Proves \Not \phi}}
{\Gamma, \Not (\phi \And \psi) \Proves \Not \phi;}
\end{equation*}

\begin{equation*}
\drv
{\drv
{\drv
{\drv
{\drv
{}
{\Gamma, \phi \And \psi \Proves \phi}}
{\Gamma, \phi \And \psi \Proves \Not (\Not \phi \And \Not (\Not \phi \And \psi)) }
&\drv
{\drv
{\drv
{\Gamma, \phi \Proves \psi}
{\Gamma, \Not(\phi \And \psi) \Proves \Not \phi}
&\drv
{\Gamma \Proves \psi
&\drv
{\drv
{}
{\Gamma, \Not (\phi \And \psi), \psi \Proves \psi}}
{\Gamma, \psi, \Not(\phi \And \psi) \Proves \psi}}
{\Gamma, \Not (\phi \And \psi) \Proves \psi}}
{\Gamma, \Not (\phi \And \psi) \Proves \Not \phi \And \psi}}
{\Gamma, \Not (\phi \And \psi) \Proves \Not (\Not \phi \And \Not (\Not \phi \And \psi))}}
{\Gamma \Proves \Not (\Not \phi \And \Not (\Not \phi \And \psi))}}
{\Gamma \Proves \Not \phi \Implies \psi}}
{\Gamma, \Not \phi \Proves \psi.}
\end{equation*}
\end{proof}

\begin{corollary}\label{4.5}
The following inference rules are derivable:
\begin{equation*}
\begin{aligned}
\drv
{\Gamma, \phi, \psi \Proves \phi}
{\Gamma, \phi, \Not \psi \Proves \phi}
\end{aligned}\;, \qquad \qquad
\begin{aligned}
\drv
{\Gamma, \phi, \psi \Proves \phi & \Gamma, \psi, \phi \Proves \psi}
{\Gamma, \Not \phi, \psi \Proves \Not \phi}
\end{aligned}\;.
\end{equation*}
\end{corollary}

\begin{proof} Refer to Proposition~\ref{2.4} and Theorems \ref{2.6} and \ref{4.4}.

\begin{equation*}
\drv
{\drv
{}
{\Gamma, \phi \Proves \phi}
&\Gamma, \phi, \psi \Proves \phi}
{\Gamma, \phi, \Not \psi \Proves \phi.}
\end{equation*}

\begin{equation*}
\drv
{\drv
{\drv
{}
{\Gamma, \Not \phi, \phi, \psi \Proves \phi}
&\drv
{}
{\Gamma, \Not \phi, \phi, \psi \Proves \Not \phi}
&\drv
{\drv
{\drv
{\drv
{\drv
{\Gamma, \psi, \phi \Proves \psi}
{\Gamma, \psi \Proves \phi \Implies \psi}}
{\Gamma\Proves \psi \Implies (\phi \Implies \psi)}
&\drv
{\drv
{\drv
{\Gamma, \phi, \psi\Proves\phi
&\drv
{}
{\Gamma, \phi, \psi \Proves \psi}}
{\Gamma, \phi, \psi, \phi \Proves \psi}}
{\Gamma, \phi, \psi \Proves \phi \Implies \psi}}
{\Gamma, \phi \Proves \psi \Implies (\phi \Implies \psi)}}
{\Gamma, \Not \phi \Proves \psi \Implies (\phi \Implies \psi)}}
{\Gamma, \Not \phi, \psi \Proves \phi \Implies \psi}}
{\Gamma, \Not \phi, \psi, \phi \Proves \psi}}
{\Gamma, \Not \phi, \psi, \phi \Proves \Not \phi}}
{\Gamma, \Not \phi, \psi \Proves \Not \phi.}
\end{equation*}
\end{proof}

\begin{corollary}\label{4.6}
The following inference rules are derivable:
\begin{equation*}
\begin{aligned}
\drv
{\Gamma \Proves \phi}
{\Gamma \Proves \phi \Or \psi}
\end{aligned}\;,
\qquad \qquad
\begin{aligned}
\drv
{\Gamma \Proves \psi}
{\Gamma \Proves \phi \Or \psi}
\end{aligned}\;,
\end{equation*}

\begin{equation*}
\begin{aligned}
\drv
{\Gamma \Proves \phi \Or \psi & \Gamma, \phi \Proves \chi & \Gamma, \psi \Proves \chi & \Gamma, \chi, \phi \Proves \chi  & \Gamma, \chi, \psi \Proves \chi}
{\Gamma \Proves \chi}
\end{aligned}\;.
\end{equation*}
\end{corollary}

\begin{proof}
The two introduction rules are derivable by Proposition~\ref{2.9}. For the derivation of the elimination rule, refer to Proposition~\ref{2.4}, Theorems \ref{2.6} and \ref{4.4}, and Lemma~\ref{4.3}.

\begin{equation*}
\drv
{\drv
{\Gamma \Proves \phi \Or \psi
&\drv
{\drv
{\drv
{\drv
{\drv
{\Gamma, \phi \Proves \chi}
{\Gamma \Proves \phi \Implies \chi}
&\drv
{\Gamma, \chi, \phi \Proves \chi}
{\Gamma, \chi \Proves \phi \Implies \chi}}
{\Gamma, \Not \chi \Proves \phi \Implies \chi}}
{\Gamma, \Not \chi, \phi \Proves \chi}}
{\Gamma, \Not \chi, \Not \Not \phi \Proves \chi}}
{\Gamma, \Not \chi, \Not \Not \phi \Proves \Not \Not \chi}
&\drv
{\drv
{\drv
{\drv
{\drv
{\Gamma, \psi \Proves \chi}
{\Gamma \Proves \psi \Implies \chi}
&\drv
{\Gamma, \chi, \psi \Proves \chi}
{\Gamma, \chi \Proves \psi \Implies \chi}}
{\Gamma, \Not \chi \Proves \psi \Implies \chi}}
{\Gamma, \Not \chi, \psi \Proves \chi}}
{\Gamma, \Not \chi, \Not \Not \psi \Proves \chi}}
{\Gamma, \Not \chi, \Not \Not \psi \Proves \Not \Not \chi}}
{\Gamma \Proves \Not \Not \chi}}
{\Gamma \Proves \chi.}
\end{equation*}
\end{proof}

\begin{proposition}\label{4.7}
The following inference rules are derivable:
\begin{equation*}
\begin{aligned}
\drv
{\Gamma, \phi, \psi \Proves \phi & \Gamma, \psi, \phi \Proves \psi}
{\Gamma \Proves \phi \Comp \psi}
\end{aligned}\;,
\end{equation*}

\begin{equation*}
\begin{aligned}
\drv
{\Gamma \Proves \phi \Comp \psi & \Gamma, \phi, \psi, \Delta \Proves \chi}
{\Gamma, \psi, \phi, \Delta \Proves \chi}
\end{aligned}\;,
\qquad\qquad
\begin{aligned}
\drv
{\Gamma \Proves \phi \Comp \psi & \Gamma, \psi, \phi, \Delta \Proves \chi}
{\Gamma, \phi, \psi, \Delta\Proves \chi}
\end{aligned}\;.
\end{equation*}

\end{proposition}

\begin{proof}
Refer to Theorem~\ref{2.6}.

\begin{equation*}
\drv
{\drv
{\drv
{\Gamma, \phi, \psi \Proves \phi}
{\Gamma, \phi \Proves \psi \Implies \phi}}
{\Gamma \Proves \phi \Implies (\psi \Implies \phi)}
&\drv
{\drv
{\Gamma, \psi, \phi \Proves \psi}
{\Gamma, \psi \Proves \phi \Implies \psi}}
{\Gamma \Proves \psi \Implies (\phi \Implies \psi)}}
{\Gamma \Proves \phi \Comp \psi.}
\end{equation*}

\begin{equation*}\drv
{\drv
{\drv
{\drv
{\Gamma \Proves \phi \Comp \psi}
{\Gamma \Proves \phi \Implies(\psi \Implies \phi)}}
{\Gamma, \phi \Proves \psi \Implies \phi}}
{\Gamma, \phi, \psi \Proves \phi}
&\Gamma, \phi, \psi, \Delta \Proves \chi
&\drv
{\drv
{\drv
{\Gamma \Proves \phi \Comp \psi}
{\Gamma \Proves \psi \Implies(\phi \Implies \psi)}}
{\Gamma, \psi \Proves \phi \Implies \psi}}
{\Gamma, \psi, \phi \Proves \psi}}
{\Gamma, \psi, \phi, \Delta \Proves \chi.}
\end{equation*}

\quad

\noindent The derivation of the third rule is entirely similar to that of the second.
\end{proof}

\begin{proposition}\label{4.8}
The following inference rules are derivable:

\begin{equation*}
\begin{aligned}
\drv
{\Gamma \Proves \phi \Comp \psi}
{\Gamma, \phi, \psi \Proves \phi}
\end{aligned}\;,
\qquad \qquad
\begin{aligned}
\drv
{\Gamma \Proves \phi \Comp \psi}
{\Gamma, \psi, \phi \Proves \psi}
\end{aligned}\;,
\qquad \qquad
\begin{aligned}
\drv
{\Gamma \Proves \phi \Comp \psi}
{\Gamma \Proves \psi \Comp \phi}
\end{aligned}\;,
\end{equation*}

\begin{equation*}
\begin{aligned}
\drv
{\Gamma \Proves \phi \Comp \psi}
{\Gamma \Proves \phi \Comp \Not \psi}
\end{aligned}\;,
\qquad \qquad
\begin{aligned}
\drv
{\Gamma \Proves \phi \Comp \Not \psi}
{\Gamma \Proves \phi \Comp \psi}
\end{aligned}\;,
\qquad \qquad
\begin{aligned}
\drv
{\Gamma \Proves \phi \Comp \psi}
{\Gamma \Proves \Not \phi \Comp \psi}
\end{aligned}\;,
\qquad \qquad
\begin{aligned}
\drv
{\Gamma \Proves \Not \phi \Comp \psi}
{\Gamma \Proves \phi \Comp \psi}
\end{aligned}\;,
\end{equation*}

\begin{equation*}
\begin{aligned}
\drv
{\phantom{\Gamma}}
{\Gamma \Proves \phi \Comp (\phi \Implies \psi)}
\end{aligned}\;,
\qquad \qquad
\begin{aligned}
\drv
{\phantom{\Gamma}}
{\Gamma \Proves \phi \Comp (\phi \And \psi)}
\end{aligned}\;,
\qquad \qquad
\begin{aligned}
\drv
{\phantom{\Gamma}}
{\Gamma \Proves \psi \Comp (\phi \And \psi)}
\end{aligned}\;.
\end{equation*}
\end{proposition}

\begin{proof}
Refer to Propositions \ref{2.4} and \ref{4.7}, Theorem~\ref{2.6}, and Corollary \ref{4.5}.
\begin{equation*}
\drv
{\Gamma \Proves \phi \Comp \psi
&\drv
{}
{\Gamma, \psi, \phi \Proves \phi}}
{\Gamma, \phi, \psi \Proves \phi;}
\qquad \qquad
\drv
{\Gamma \Proves \phi \Comp \psi
&\drv
{}
{\Gamma, \phi, \psi \Proves \psi}}
{\Gamma, \psi, \phi \Proves \psi.}
\qquad \qquad
\drv
{\drv
{\Gamma \Proves \phi \Comp \psi}
{\Gamma, \psi, \phi \Proves \psi}
&\drv
{\Gamma \Proves \phi \Comp \psi}
{\Gamma, \phi, \psi \Proves \phi}}
{\Gamma \Proves \psi \Comp \phi.}
\end{equation*}
\begin{equation*}
\drv
{\drv
{\drv
{\Gamma \Proves \phi \Comp \psi}
{\Gamma, \phi, \psi \Proves \phi}}
{\Gamma, \phi, \Not \psi \Proves \phi}
&\drv
{\drv
{\Gamma \Proves \phi \Comp \psi}
{\Gamma, \psi, \phi \Proves \psi}
&\drv
{\Gamma \Proves \phi \Comp \psi}
{\Gamma, \phi, \psi \Proves \phi}}
{\Gamma, \Not \psi, \phi \Proves \Not \psi}}
{\Gamma \Proves \phi \Comp \Not \psi;}
\qquad \qquad
\drv
{\drv
{\drv
{\drv
{\Gamma \Proves \phi \Comp \Not \psi}
{\Gamma \Proves \phi \Comp \Not \Not \psi}}
{\Gamma, \phi, \Not \Not \psi \Proves \phi}}
{\Gamma, \phi, \psi \Proves\phi}
&\drv
{\drv
{\drv
{\drv
{\Gamma \Proves \phi \Comp \Not \psi}
{\Gamma \Proves \phi \Comp \Not \Not \psi}}
{\Gamma, \Not \Not \psi, \phi \Proves \Not \Not \psi}}
{\Gamma, \Not \Not \psi, \phi \Proves \psi}}
{\Gamma, \psi, \phi \Proves \psi}}
{\Gamma \Proves \phi \Comp \psi.}
\end{equation*}

\noindent The derivations of the two remaining rules for negation are entirely similar.

\begin{equation*}
\drv
{\drv
{\drv
{\drv
{\drv
{}
{\Gamma, \phi \Implies \psi, \Not \phi, \phi \Proves \psi}}
{\Gamma, \phi \Implies \psi, \Not \phi \Proves \phi \Implies \psi}
&\drv
{\drv
{\drv
{}
{\Gamma, \Not \phi, \phi \Proves \psi}}
{\Gamma, \Not \phi \Proves \phi \Implies \psi}
&\drv
{}
{\Gamma, \Not \phi \Proves \Not \phi}}
{\Gamma, \Not \phi, \phi \Implies \psi \Proves \Not \phi}}
{\Gamma \Proves (\phi \Implies \psi)\Comp \Not \phi}}
{\Gamma \Proves (\phi \Implies \psi)\Comp \phi}}
{\Gamma \Proves \phi \Comp (\phi \Implies \psi).}
\end{equation*}

\begin{equation*}
\drv
{\drv
{\drv
{}
{\Gamma, \phi, \phi \And \psi \Proves \phi \And \psi}}
{\Gamma, \phi, \phi \And \psi \Proves \phi}
&\drv
{\drv
{\drv
{}
{\Gamma, \phi \And \psi \Proves \phi \And \psi}}
{\Gamma, \phi \And \psi \Proves \phi}
&\drv
{}
{\Gamma, \phi \And \psi \Proves \phi \And \psi}}
{\Gamma, \phi \And \psi, \phi \Proves \phi \And \psi}}
{\Gamma \Proves \phi \Comp (\phi \And \psi).}
\end{equation*}
\noindent The derivation of the remaining rule for conjunction is entirely similar.
\end{proof}

\begin{proposition}\label{4.9}
The following inference rule is derivable:
\begin{equation*}
\begin{aligned}
\drv
{\Gamma \Proves \phi \Comp \psi & \Gamma, \phi \Proves \psi}
{\Gamma, \Not \psi \Proves \Not \phi}
\end{aligned}\;.
\end{equation*}
\end{proposition}

\begin{proof}
Refer to Theorem~\ref{2.6} and Proposition~\ref{4.8}.
\begin{equation*}
\drv
{\drv
{\Gamma, \phi \Proves \psi
&\drv
{}
{\Gamma, \phi, \psi, \Not \psi \Proves \Not \phi}}
{\Gamma, \phi, \Not \psi \Proves \Not \phi}
&\drv
{\drv
{\drv
{\Gamma \Proves \phi \Comp \psi}
{\Gamma \Proves \phi \Comp \Not \psi}}
{\Gamma \Proves \Not \phi \Comp \Not \psi}}
{\Gamma, \Not \phi, \Not \psi \Proves \Not \phi}}
{\Gamma, \Not \psi \Proves \Not \phi.}
\end{equation*}
\end{proof}

\begin{lemma}\label{4.10}
The following inference rules are derivable:
\begin{equation*}
\begin{aligned}
\drv
{\Gamma \Proves \phi \Comp \psi}
{\Gamma, \phi, \psi \Proves \phi \And \psi}
\end{aligned}\;,
\qquad \qquad
\begin{aligned}
\drv
{\Gamma \Proves \phi \Comp \psi}
{\Gamma, \phi, \Not \psi \Proves \phi \And \Not \psi}
\end{aligned}\;,
\end{equation*}

\begin{equation*}
\begin{aligned}
\drv
{\Gamma \Proves \phi \Comp \psi}
{\Gamma, \Not \phi, \psi \Proves \Not \phi \And \psi}
\end{aligned}\;,
\qquad \qquad
\begin{aligned}
\drv
{\Gamma \Proves \phi \Comp \psi}
{\Gamma, \Not \phi, \Not \psi \Proves \Not \phi \And \Not \psi}
\end{aligned}\;.
\end{equation*}
\end{lemma}

\begin{proof}
Refer to Proposition~\ref{4.8}.
\begin{equation*}
\drv
{\drv
{\Gamma \Proves \phi \Comp \psi}
{\Gamma, \phi, \psi \Proves \phi}
&\drv
{}
{\Gamma, \phi, \psi \Proves \psi}}
{\Gamma, \phi, \psi \Proves\phi \And \psi;}
\qquad \qquad
\drv
{\drv
{\drv
{\Gamma \Proves \phi \Comp \psi}
{\Gamma \Proves \phi \Comp \Not \psi}}
{\Gamma, \phi, \Not \psi \Proves \phi}
&\drv
{}
{\Gamma, \phi, \Not \psi \Proves \Not \psi}}
{\Gamma, \phi, \Not \psi \Proves \phi \And \Not \psi;}
\end{equation*}

\begin{equation*}
\drv
{\drv
{\drv
{\Gamma \Proves \phi \Comp \psi}
{\Gamma \Proves \Not \phi \Comp \psi}}
{\Gamma, \Not \phi, \psi \Proves \Not \phi}
&\drv
{}
{\Gamma, \Not \phi, \psi \Proves \psi}}
{\Gamma, \Not \phi, \psi \Proves \Not\phi \And \psi;}
\qquad \qquad
\drv
{\drv
{\drv
{\drv
{\Gamma \Proves \phi \Comp \psi}
{\Gamma \Proves \phi \Comp \Not \psi}}
{\Gamma \Proves \Not \phi \Comp \Not \psi}}
{\Gamma \Proves \Not \phi, \Not \psi \Proves \Not \phi}
&\drv
{}
{\Gamma \Proves \Not \phi, \Not \psi \Proves \Not \psi}}
{\Gamma \Proves \Not \phi, \Not \psi \Proves \Not \phi \And \Not \psi.}
\end{equation*}

\end{proof}

\begin{proposition}\label{4.11}
The following inference rules are derivable:
\begin{equation*}
\begin{aligned}
\drv
{\Gamma \Proves \phi \Comp \psi}
{\Gamma, \phi \Proves (\phi \And \psi) \Or (\phi \And \Not \psi)}
\end{aligned}\;,
\end{equation*}

\begin{equation*}
\begin{aligned}
\drv
{\Gamma \Proves \phi \Comp \psi}
{\Gamma \Proves ((\phi \And \psi) \Or (\phi \And \Not \psi)) \Or ((\Not\phi \And \psi) \Or (\Not\phi \And \Not \psi))}
\end{aligned}\;.
\end{equation*}
\end{proposition}

\begin{proof}
Refer to Corollary \ref{4.6}, Proposition~\ref{4.8}, and Lemma~\ref{4.10}.

\begin{equation*}
\drv
{\drv
{\drv
{\Gamma \Proves \phi \Comp \psi}
{\Gamma, \phi, \psi \Proves \phi \And \psi}}
{\Gamma, \phi,\psi \Proves (\phi \And \psi) \Or (\phi \And \Not \psi)}
&\drv
{\drv
{\Gamma \Proves \phi \Comp \psi}
{\Gamma, \phi, \Not \psi \Proves \phi \And \Not \psi}}
{\Gamma, \phi, \Not \psi \Proves (\phi \And \psi) \Or (\phi \And \Not \psi)}}
{\Gamma, \phi \Proves (\phi \And \psi) \Or (\phi \And \Not \psi).}
\end{equation*}

\quad

\noindent Let $\chi_0$ be the formula $((\phi \And \psi) \Or (\phi \And \Not \psi)) \Or ((\Not\phi \And \psi) \Or (\Not\phi \And \Not \psi))$.

\begin{equation*}
\drv{
\drv
{\drv
{\Gamma \Proves \phi \Comp \psi}
{\Gamma, \phi \Proves (\phi \And \psi) \Or (\phi \And \Not \psi)}}
{\Gamma, \phi \Proves \chi_0}
&\drv
{\drv
{\drv
{\Gamma \Proves \phi \Comp \psi}
{\Gamma \Proves \Not \phi \Comp \psi}}
{\Gamma, \Not \phi \Proves (\Not \phi \And \psi) \Or (\Not \phi \And \Not \psi)}}
{\Gamma, \Not \phi \Proves \chi_0}}
{\Gamma \Proves \chi_0.}
\end{equation*}
\end{proof}

\begin{lemma}\label{4.12}
The following inference rules are derivable:
\begin{equation*}
\begin{aligned}
\drv
{\Gamma \Proves \phi \And \psi}
{\Gamma \Proves \phi \Comp \psi}
\end{aligned}\;,
\qquad \qquad
\begin{aligned}
\drv
{\Gamma \Proves \phi \And \Not \psi}
{\Gamma \Proves \phi \Comp \psi}
\end{aligned}\;,
\end{equation*}

\begin{equation*}
\begin{aligned}
\drv
{\Gamma\Proves \Not \phi \And \psi}
{\Gamma \Proves \phi \Comp \psi}
\end{aligned}\;,
\qquad \qquad
\begin{aligned}
\drv
{\Gamma \Proves \Not \phi \And \Not \psi}
{\Gamma \Proves \phi \Comp \psi}
\end{aligned}\;.
\end{equation*}
Consequently, the following sequents are derivable:
\begin{enumerate}
\item $\Gamma, \phi \And \psi \Proves \phi \Comp \psi$,
\item $\Gamma, \phi \And \Not \psi \Proves \phi \Comp \psi$,
\item $\Gamma, \Not \phi \And \psi \Proves \phi \Comp \psi$,
\item $\Gamma, \Not \phi \And \Not \psi \Proves \phi \Comp \psi$.
\end{enumerate}
\end{lemma}

\begin{proof}
Refer to Propositions \ref{4.7} and \ref{4.8}.
\begin{equation*}
\drv
{\drv
{\Gamma \Proves \phi \And \psi
&\drv
{\drv
{\drv
{}
{\Gamma, \phi \And \psi \Proves \phi \And \psi}}
{\Gamma, \phi \And \psi \Proves \phi}
&\
\drv
{\drv
{}
{\Gamma, \phi \And \psi \Proves \phi \And \psi}}
{\Gamma, \phi \And \psi \Proves \psi}}
{\Gamma, \phi \And \psi, \phi \Proves \psi}}
{\Gamma, \phi \Proves \psi}
&\drv
{}
{\Gamma, \phi \Proves \phi}}
{\Gamma, \phi, \psi \Proves \phi.}
\end{equation*}
The derivation of the rule \;$\begin{aligned}\drv{\Gamma \Proves \phi \And \psi}{\Gamma, \psi, \phi \Proves \psi} \end{aligned}$\; is entirely similar.
\begin{equation*}
\drv
{\drv
{\Gamma \Proves \phi \And \psi}
{\Gamma, \phi, \psi \Proves \phi}
&\drv
{\Gamma \Proves \phi \And \psi}
{\Gamma, \psi, \phi \Proves \psi}}
{\Gamma \Proves \phi \Comp \psi;}
\qquad \qquad
\drv
{\drv
{\Gamma \Proves \phi \And \Not \psi}
{\Gamma \Proves \phi \Comp \Not \psi}}
{\Gamma \Proves \phi \Comp \psi;}
\qquad \qquad
\drv
{\drv
{\Gamma \Proves \Not \phi \And \psi}
{\Gamma \Proves \Not \phi \Comp \psi }}
{\Gamma \Proves \phi \Comp \psi;}
\qquad \qquad
\drv
{\drv
{\drv
{\Gamma \Proves \Not \phi \And \Not \psi}
{\Gamma \Proves \Not \phi \Comp \Not \psi }}
{\Gamma \Proves \phi \Comp \Not \psi}}
{\Gamma \Proves \phi \Comp \psi.}
\end{equation*}
\end{proof}

\begin{lemma}\label{4.13}
The following inference rule is derivable:
\begin{equation*}
\begin{aligned}
\drv
{\Gamma \Proves (\phi \And \psi) \Or (\phi \And \Not \psi)}
{\Gamma \Proves \phi \Comp \psi}
\end{aligned}\;.
\end{equation*}
Consequently, the following sequents are derivable:
\begin{enumerate}
\item $\Gamma, (\phi \And \psi) \Or (\phi \And \Not \psi) \Proves \phi \Comp \psi$ ,
\item $\Gamma, (\Not \phi \And \psi) \Or (\Not \phi \And \Not \psi) \Proves \phi \Comp \psi$.
\end{enumerate}
\end{lemma}

\begin{proof}
Refer to Corollary \ref{4.6}, Proposition~\ref{4.8} and Lemma~\ref{4.12}. Let $\xi_1$ be the formula $\phi \And \psi$, and let $\xi_2$ be the formula $\phi \And \Not \psi$.
\begin{equation*}
\drv
{\Gamma \Proves \xi_1 \Or \xi_2
&
\drv
{}
{\Gamma, \xi_1 \Proves \phi \Comp \psi}
&
\drv
{}
{\Gamma, \xi_2 \Proves \phi \Comp \psi}
&\drv
{}
{\Gamma, \phi \Comp \psi, \xi_1 \Proves \phi \Comp \psi}
&\drv
{}
{\Gamma, \phi \Comp \psi, \xi_2 \Proves \phi \Comp \psi}}
{\Gamma \Proves \phi \Comp \psi.}
\end{equation*}

\noindent It follows that sequent (1) is derivable. So is sequent (2):
\begin{equation*}
\drv{
\drv
{\drv{}
{\Gamma, (\Not \phi \And \psi) \Or (\Not \phi \And \Not \psi) \Proves (\Not \phi \And \psi) \Or (\Not \phi \And \Not \psi)}}
{\Gamma, (\Not \phi \And \psi) \Or (\Not \phi \And \Not \psi) \Proves \Not \phi \Comp  \psi}}
{\Gamma, (\Not \phi \And \psi) \Or (\Not \phi \And \Not \psi) \Proves \phi \Comp  \psi.}
\end{equation*}
\end{proof}

\begin{proposition}\label{4.14}
The following inference rule is derivable:
\begin{equation*}
\begin{aligned}
\drv
{\Gamma \Proves ((\phi \And \psi) \Or (\phi \And \Not \psi)) \Or ((\Not \phi \And \psi) \Or (\Not \phi \And \Not \psi))}
{\Gamma\Proves \phi \Comp \psi}
\end{aligned}\;.
\end{equation*}
\end{proposition}

\begin{proof}
Refer to Corollary \ref{4.6} and Lemma~\ref{4.13}. Let $\zeta_1$ be the formula $(\phi \And \psi) \Or (\phi \And \Not \psi)$, and let $\zeta_2$ be the formula $(\Not \phi \And \psi) \Or (\Not \phi \And \Not \psi)$.
\begin{equation*}
\drv
{\Gamma \Proves \zeta_1 \Or \zeta_2
&
\drv
{}
{\Gamma, \zeta_1 \Proves \phi \Comp \psi}
&
\drv
{}
{\Gamma, \zeta_2 \Proves \phi \Comp \psi}
&\drv
{}
{\Gamma, \phi \Comp \psi, \zeta_1 \Proves \phi \Comp \psi}
&\drv
{}
{\Gamma, \phi \Comp \psi, \zeta_2 \Proves \phi \Comp \psi}}
{\Gamma \Proves \phi \Comp \psi.}
\end{equation*}
\end{proof}

\section{Universal and existential quantification}\label{S5}

This section extends the propositional deductive system $NOM$ to two predicate deductive systems. The first is shown to be sound for Takeuti's semantics \cite{Takeuti}, and the second is shown to be sound for a variant of Weaver's semantics \cite{Weaver}. For simplicity, we identify formulas that differ only in the symbols used for their bound variables so that any term is substitutable for any variable in any formula.

Throughout this section, let $\Q$ be a complete orthomodular lattice, and let $V^{(\Q)}$ be the $\Q$-valued universe \cite{Takeuti}\cite{Ozawa5}*{sec.~4.1}.

\begin{definition}\label{5.1}
We define the single-sorted predicate deductive system $NOM_\Q$:
\begin{enumerate}
\item Its formulas are the first-order formulas whose 
\begin{enumerate}
\item connectives are $\And$, $\Implies$, and $\Not$;
\item quantifiers are $\forall$;
\item relation symbols are $=$ and $\in$;
\item constant symbols are the elements of $V^{(\Q)}$, with no other function symbols.
\end{enumerate}
\item Its rules of inference are those in Figure~\ref{F1} together with
\begin{equation*}
\begin{aligned}
\drv
{\Gamma \Proves \phi}
{\Gamma \Proves (\forall x)\phi}
\end{aligned}\;,
\qquad \qquad
\begin{aligned}
\drv
{\Gamma \Proves (\forall x)\phi}
{\Gamma \Proves \phi[x/t]}
\end{aligned}\;,
\end{equation*}
where the former rule is subject to the standard constraint that the variable $x$ must not appear freely in $\Gamma$.
\end{enumerate}
\end{definition}

\begin{corollary}\label{5.2}
Let $\[\cdot\]_\Q$ be a $\Q$-valued interpretation of the formulas of $NOM_\Q$ whose quantized implication is the Sasaki arrow \cite{Ozawa5}*{sec.~4.2}, e.g., Takeuti's interpretation \cite{Takeuti}. Let $\phi$ be a closed formula of $NOM_\Q$. If $\Proves \phi$ is derivable in $NOM_\Q$, then $\[\phi\]_\Q = \top$.
\end{corollary}

\begin{proof}
We define a sequent $\phi_1(x_1, \ldots, x_m), \ldots, \phi_n(x_1, \ldots, x_m) \Proves \psi (x_1, \ldots, x_m)$ to be \emph{$\Q$-true} if $\[\phi_1(u_1, \ldots, u_m)\]_\Q \Andthen \cdots \Andthen \[\phi_n(u_1, \ldots, u_m)\]_\Q \leq \[\psi(u_1, \ldots, u_m)\]_\Q$ for all $u_1, \ldots, u_m \in V^{(\Q)}$. We prove by induction that every derivable sequent of $NOM_\Q$ is $\Q$-true. It is evidently sufficient to verify that each rule is \emph{$\Q$-sound}: if its premises are $\Q$-true, then its conclusion is also $\Q$-true. For those rules that appear in Figure~\ref{F1}, this follows by Theorem~\ref{3.5}.

Let $\phi_1(x_1, \ldots, x_m), \ldots, \phi_n(x_1, \ldots, x_m) \Proves \psi (x_0, x_1, \ldots, x_m)$ be a $\Q$-true sequent. Then, \begin{align*}&\[\phi_1(u_1, \ldots, u_m)\]_\Q \Andthen \cdots \Andthen \[\phi_n(u_1, \ldots, u_m)\]_\Q \\ & \hspace{20ex}\leq \bigwedge_{u \in V^{(\Q)}} \[\psi(u, u_1, \ldots, u_m)\]_\Q
= \[(\forall x_0)\psi(x_0, u_1, \ldots, u_m)\]_\Q
\end{align*}
for all $u_1, \ldots, u_m \in V^{(\Q)}$. Therefore, the sequent $\phi_1(x_1, \ldots, x_m), \ldots, \phi_n(x_1, \ldots, x_m) \Proves (\forall x_0) \psi (x_0, x_1, \ldots, x_m)$ is $\Q$-true, and more generally, the $\forall$-introduction rule is $\Q$-sound.

Let $\phi_1(x_1, \ldots, x_m), \ldots, \phi_n(x_1, \ldots, x_m) \Proves (\forall x_0) \psi (x_0, x_1, \ldots, x_m)$ be a $\Q$-true sequent, and let $t$ be a term. We calculate that 
\begin{align*}&
\[\phi_1(u_1, \ldots, u_m)\]_\Q \Andthen \cdots \Andthen \[\phi_n(u_1, \ldots, u_m)\]_\Q
\\ & \hspace{20ex} \leq 
\[(\forall x_0) \phi_n(x_0, u_1 \ldots, u_m)\]_\Q
=
\bigwedge_{u \in V^{(\Q)}} \[\phi_n(u, u_1, \ldots, u_m)\]_\Q
\end{align*}
for all $u_1, \ldots, u_m \in V^{(\Q)}$. If $t$ is a variable, then $\phi_1(x_1, \ldots, x_m), \ldots, \phi_n(x_1, \ldots, x_m) \Proves \psi (t, x_1, \ldots, x_m)$ is $\Q$-true because $\bigwedge_{u \in V^{(\Q)}} \[\phi_n(u, u_1, \ldots, u_m)\]_\Q \leq \[\phi_n(u_0, u_1, \ldots, u_m)\]_\Q$ for all $u_0 \in V^{(\Q)}$, and if $t$ is a constant, then $\phi_1(x_1, \ldots, x_m), \ldots, \phi_n(x_1, \ldots, x_m) \Proves \psi (t, x_1, \ldots, x_m)$ is $\Q$-true for the same reason. Therefore, the $\forall$-elimination rule is $\Q$-sound.

Altogether, we find that all of the inference rules of $NOM_\Q$ are $\Q$-sound. In particular, if $\phi$ is a closed formula such that $\Proves \phi$ is derivable in $NOM_\Q$, then $\Proves \phi$ is $\Q$-true, and thus $\top \leq \[\phi\]_\Q$ as desired.
\end{proof}

The quantum sets that are the discrete quantum spaces of noncommutative geometry \cite{quantumsets}*{Def.~2.1}\cite{PodlesWoronowicz} form a category in two inequivalent natural ways. The category $\cat{qSet}$ of quantum sets and functions generalizes the category $\cat{Set}$ of sets and functions, and it is dual to the category of hereditarily atomic von Neumann algebras and unital normal $*$-homomorphisms \cite{quantumsets}*{Def.~5.3, Thm.~7.4}. The category $\cat{qRel}$ of quantum sets and binary relations generalizes the category $\cat{Rel}$ of sets and binary relations, and it is dual to the category of hereditarily atomic von Neumann algebras and quantum relations in the sense of Weaver \cite{discretequantumstructures}*{app.~A.2}\cite{Weaver2}.

A function between quantum sets is formally a kind of binary relation between quantum sets \cite{quantumsets}*{Def.~4.1}; $\cat{qSet}$ is a subcategory of $\cat{qRel}$. The two categories share the same symmetric monoidal structure, whose product is notated $\times$ because it generalizes the Cartesian product of sets in a suitable way \cite{quantumsets}*{Defs.~2.2,~3.3}. The unit of the monoidal structure is notated $\mathbf{1}$.
For a quantum set $\X$, there exist suitable projection functions $\X \times \X \to \X$, but in general, there does not exist a suitable diagonal function $\X \to \X \times \X$ \cite{quantumsets}*{Def.~10.3}. This is not a defect in the definition but rather a basic feature of quantum spaces in noncommutative geometry \cite{Woronowicz}.

In this setting, we work with a class of formulas that does not presume the existence of diagonal functions. A \emph{nonduplicating} formula is a formula such that no variable, bound or free, appears more than once in any atomic subformula. For illustration, $P(x,y) \And Q(x,y)$ is a nonduplicating formula, and $P(x,x) \And Q(y,y)$ is not a nonduplicating formula.

\begin{definition}\label{5.3}
We define the many-sorted predicate deductive system $NOM_\qqq$:
\begin{enumerate}
\item Its formulas are the nonduplicating first-order formulas whose
\begin{enumerate}
\item connectives are $\And$, $\Implies$, and $\Not$;
\item quantifiers are $\forall$;
\item sorts are quantum sets;
\item relation symbols of each arity $(\X_1, \ldots, \X_n)$ are binary relations $\X_1 \times \cdots \times \X_n \to \mathbf 1;$
\item function symbols of each arity $(\X_1, \ldots, \X_m; \Y)$ are functions $\X_1 \times \cdots \times \X_m \to \Y.$
\end{enumerate}
\item Its rules of inference are those in Figure~\ref{F1} together with
\begin{equation*}
\begin{aligned}
\drv
{\Gamma \Proves \phi}
{\Gamma \Proves (\forall x)\phi}
\end{aligned}\;,
\qquad \qquad
\begin{aligned}
\drv
{\Gamma \Proves (\forall x)\phi}
{\Gamma \Proves \phi[x/t]}
\end{aligned}\;,
\qquad
\qquad
\begin{aligned}
\drv
{\Gamma, \phi, \psi \Proves \chi}
{\Gamma, \psi, \phi \Proves \chi}
\end{aligned}\;,
\end{equation*}
where the first rule is subject to the constraint that $x$ must not appear freely in $\Gamma$, the second rule is subject to the constraint that $\phi$ and $t$ must not have any free variables in common, and the third rule is subject to the constraint that $\phi$ and $\psi$ must not have any free variables in common.
\end{enumerate}
\end{definition}

\begin{corollary}\label{5.4}
Let $\[\cdot\]_\qqq$ be Weaver's interpretation of the formulas of $NOM_\qqq$ \cite{discretequantumstructures}*{sec.~2}\cite{Weaver} with implication being interpreted as the Sasaki arrow. Let $\phi$ be a closed formula of $NOM_\qqq$. If $\Proves \phi$ is derivable in $NOM_\qqq$, then $\[\phi\]_\qqq = \top$.
\end{corollary}

\begin{proof}
We define a sequent $\phi_1, \ldots, \phi_n \Proves \psi$ to be \emph{$\qqq$-true} if $\[\bar x \suchthat \phi_1\]_\qqq \Andthen \cdots \Andthen \[\bar x \suchthat \phi_n\]_\qqq \leq \[ \bar x \suchthat \psi\]_\qqq$ for all tuples of distinct variables $\bar x = (x_1, \ldots, x_m)$ that include the free variables of the sequent. We prove by induction that every derivable sequent in $NOM_\qqq$ is $\qqq$-true. It is evidently sufficient to verify that each rule is \emph{$\qqq$-sound}: if its premises are $\qqq$-true, then its conclusion is $\qqq$-true. For those rules that appear in Figure 1, this follows by Theorem~\ref{3.5}; the binary relations between any two quantum sets form an orthomodular lattice \cite{quantumsets}*{Def.~3.8}.

Let $\phi_1, \ldots, \phi_n \Proves \psi$ be a $\qqq$-true sequent, and let $x_0$ be a variable that does not appear freely in the antecedent. Let $\overline x = (x_1, \ldots, x_m)$ be a tuple of distinct variables that includes the free variables of the sequent $\phi_1, \ldots, \phi_n \Proves (\forall x_0)\psi$. Without loss of generality, we may assume that $x_0$ does not appear in $\overline x$. Because the sequent $\phi_1, \ldots, \phi_n \Proves \psi$ is $\qqq$-true, we have that $\[x_0, \overline x \suchthat  \phi_1\]_\qqq \Andthen \cdots \Andthen \[x_0, \overline x  \suchthat  \phi_n\]_\qqq \leq \[x_0, \overline x \suchthat  \psi\]_\qqq$. By \cite{discretequantumstructures}*{Prop.~2.3.3}, $\[x_0, \overline x \suchthat  \phi_i\]_\qqq = \top_{\X_0} \times \[\overline x, \phi_i\]_\qqq$ for each index $i \in \{1, \ldots, n\}$, where $\X_0$ is the sort of $x_0$. We calculate that
\begin{align*} 
\top_{\X_0} \times (\[\overline x \suchthat  \phi_1\]_\qqq \Andthen \cdots \Andthen \[\overline x  \suchthat  \phi_n\]_\qqq)
& =
(\top_{\X_0} \times \[\overline x \suchthat  \phi_1\]_\qqq) \Andthen \cdots \Andthen 
(\top_{\X_0} \times \[\overline x \suchthat  \phi_n\]_\qqq)
\\ & =
\[x_0, \overline x \suchthat  \phi_1\]_\qqq \Andthen \cdots \Andthen \[x_0, \overline x  \suchthat  \phi_n\]_\qqq
\leq
\[x_0, \overline x \suchthat  \psi\]_\qqq,
\end{align*}
which implies that $\[\overline x \suchthat  \phi_1\]_\qqq \Andthen \cdots \Andthen \[\overline x  \suchthat  \phi_n\]_\qqq \leq \[\overline x  \suchthat (\forall x_0) \psi) \]_\qqq$ by \cite{discretequantumstructures}*{Def.~2.3.2}. Therefore, $\phi_1, \ldots, \phi_n \Proves (\forall x_0) \psi$ is $\qqq$-true, and more generally, the $\forall$-introduction rule is $\qqq$-sound.

Let $\phi_1, \ldots, \phi_n \Proves (\forall x_0)\psi$ be a $\qqq$-true sequent, and let $t$ be a term of the same sort as $x_0$.  Assume that $\psi$ and $t$ have no free variables in common. Let $\overline x = (x_1, \ldots, x_m)$ be a tuple of distinct variables that includes the free variables of $\phi_1, \ldots, \phi_n \Proves \psi[x_0/t]$. Without loss of generality, we may assume that $x_0$ does not appear in $\overline x$. Let $\overline y = (y_1, \ldots, y_k)$ and $\overline z = (z_1, \ldots, z_\ell)$ be tuples of distinct variables such that the variables of $t$ appear in $\overline y$, the free variables of $\psi$ appear in $\overline z$, and the concatenation of $\overline y$ and $\overline z$ is a permutation of $\overline x$. In particular $\overline y$ and $\overline z$ have no variables in common, and neither include $x_0$.

Let $\X_0$ be the sort of $x_0$, let $\Y_1, \ldots, \Y_k$ be the sorts of $y_1, \ldots, y_k$, respectively, and let $\Z_1, \ldots, \Z_\ell$ be the sorts of $z_1, \ldots, z_\ell$, respectively. Let $\Y = \Y_1 \times \cdots \times \Y_k$, and let $\Z = \Z_1 \times \cdots \times \Z_\ell$. We calculate that
\begin{align*} & 
\[\overline y, \overline z \suchthat \phi_1\]_\qqq \Andthen \cdots \Andthen \[\overline y, \overline z  \suchthat \phi_n\]_\qqq
\leq
\[ \overline y, \overline z \suchthat (\forall x_0) \psi \]_\qqq
=
\top_\Y \times \[ \overline z \suchthat (\forall x_0) \psi\]_\qqq
\\ & = 
(\top_{\X_0} \circ \[ \overline y \suchthat t\]_\qqq) \times ( \[\overline z \suchthat  (\forall x_0) \psi\]_\qqq \circ I_\Z)
= 
(\top_{\X_0} \times \[\overline z \suchthat  (\forall x_0) \psi\]_\qqq) \circ 
(\[ \overline y \suchthat t\]_\qqq \times I_\Z)
\\ & \hspace{30ex} \leq 
\[x_0, \overline z \suchthat \psi\]_\qqq \circ (\[ \overline y \suchthat t\]_\qqq \times I_\Z)
=
\[ \overline y, \overline z \suchthat \psi[x_0/t]\]_\qqq,
\end{align*}
where $\top_\Y = \top_{\X_0} \circ \[ \overline y \suchthat t\]_\qqq$ because $\[ \overline y \suchthat t\]_\qqq$ is a function by \cite{discretequantumstructures}*{Lem.~3.5.3} and $\[x_0, \overline x \suchthat \psi\]_\qqq \circ (\[ \overline y \suchthat t\]_\qqq \times I_\Z)
=
\[ \overline y, \overline x \suchthat \psi[x_0/t]\]_\qqq$ by Proposition~\ref{B.2}. Indeed, any function $F$ from a quantum set $\Y$ to a quantum set $\X_0$ satisfies $\top_\Y = \top_{\X_0} \circ F$ because $\top_\Y = \top_\Y \circ I_\Y \leq  \top_\Y \circ F^\dagger \circ F \leq \top_{\X_0} \circ F $. We conclude that $\[\overline y, \overline z \suchthat \phi_1\]_\qqq \Andthen \cdots \Andthen \[\overline y, \overline z  \suchthat \phi_n\]_\qqq \leq \[ \overline y, \overline z \suchthat \psi[x_0/t]\]_\qqq$. It follows by \cite{discretequantumstructures}*{Def.~2.2.4, Prop.~2.3.3} that $\[\overline x \suchthat \phi_1\]_\qqq \Andthen \cdots \Andthen \[\overline x  \suchthat \phi_n\]_\qqq \leq \[ \overline x \suchthat \psi[x_0/t]\]_\qqq$. Therefore, the sequent $\phi_1, \ldots, \phi_n \Proves \psi[x_0/t]$ is $\qqq$-true, and more generally, the $\forall$-elimination is $\qqq$-sound.

Let $\phi_1, \ldots, \phi_n, \phi, \psi \Proves \chi$ be a $\qqq$-true sequent, and assume that $\phi$ and $\psi$ have no variables in common. Let $\overline x$ be a tuple of distinct variables that includes the free variables of $\phi_1, \ldots, \phi_n, \psi, \phi \Proves \chi$. Let $\overline y = (y_1, \ldots, y_k)$ and $\overline z = (z_1, \ldots, z_\ell)$ be tuples of distinct variables such that the free variables of $\phi$ appear in $\overline y$, the free variables of $\psi$ appear in $\overline z$, and the concatentation of $\overline y$ and $\overline z$ is a permutation of $\overline x$. Let $\Y_1, \ldots, \Y_k$ be the sorts of $y_1, \ldots, y_k$, and let $\Z_1, \ldots, \Z_\ell$ be the sorts of $z_1, \ldots, z_\ell$. Let $\Y = \Y_1 \times \cdots \times \Y_k$, and let $\Z = \Z_1 \times \cdots \times \Z_\ell$. Applying \cite{discretequantumstructures}*{Prop.~2.3.3}, we calculate that
\begin{align*}
\[\overline y, \overline z \suchthat \phi_1\]_\qqq \Andthen & \cdots \Andthen \[\overline y, \overline z \suchthat \phi_n\]_\qqq \Andthen \[\overline y, \overline z \suchthat \psi\]_\qqq \Andthen \[\overline y, \overline z \suchthat \phi\]_\qqq
\\ & =
\[\overline y, \overline z \suchthat \phi_1\]_\qqq \Andthen \cdots \Andthen \[\overline y, \overline z \suchthat \phi_n\]_\qqq \Andthen (\top_\Y \times \[\overline z \suchthat \psi\]_\qqq) \Andthen (\[\overline y \suchthat \phi\]_\qqq \times \top_\Z)
\\ & =
\[\overline y, \overline z \suchthat \phi_1\]_\qqq \Andthen \cdots \Andthen \[\overline y, \overline z \suchthat \phi_n\]_\qqq \Andthen (\[\overline y \suchthat \phi\]_\qqq \times \[\overline z \suchthat \psi\]_\qqq)
\\ & = 
\[\overline y, \overline z \suchthat \phi_1\]_\qqq \Andthen \cdots \Andthen \[\overline y, \overline z \suchthat \phi_n\]_\qqq \Andthen (\[\overline y \suchthat \phi\]_\qqq \times \top_\Z) \Andthen ( \top_\Y \times \[\overline z \suchthat \psi\]_\qqq)
\\ &
=
\[\overline y, \overline z \suchthat \phi_1\]_\qqq \Andthen \cdots \Andthen \[\overline y, \overline z \suchthat \phi_n\]_\qqq \Andthen \[\overline y, \overline z \suchthat \phi\]_\qqq \Andthen \[\overline y, \overline z \suchthat \psi\]_\qqq 
\leq
\[ \overline y, \overline z \suchthat \chi\]_\qqq,
\end{align*}
where the first and fourth equalities follow by \cite{discretequantumstructures}*{Prop.~2.3.3} and 
the second and third equalities follow by linear algebra \cite{quantumsets}*{Def.~3.8}.
It follows by \cite{discretequantumstructures}*{Def.~2.2.4, Prop.~2.3.3} that $\[\overline x \suchthat \phi_1\]_\qqq \Andthen \cdots \Andthen \[\overline x \suchthat \phi_n\]_\qqq \Andthen \[\overline x \suchthat \psi\]_\qqq \Andthen \[\overline x \suchthat \phi\]_\qqq \leq \[ \overline x \suchthat \chi\]_\qqq$.
Therefore, the sequent $\phi_1, \ldots, \phi_n, \psi, \phi \Proves \chi$ is $\qqq$-true, and more generally, the third rule of inference in Definition~\ref{5.3} is $\qqq$-sound.

Altogether, we have that all of the inference rules of $NOM_\qqq$ are $\qqq$-sound. In particular, if $\phi$ is a closed formula such that $\Proves \phi$ is derivable in $NOM_\qqq$, then $\Proves \phi$ is $\qqq$-true.
\end{proof}

\begin{definition}\label{5.5}
Let $(\exists x) \phi$ be an abbreviation for $\Not (\forall x) \Not \phi$.
\end{definition}

\begin{lemma}\label{5.6}
The sequent $\Gamma \Proves (\forall x) \phi \Comp \phi[x/t]$ is derivable in $NOM_\Q$. It is also derivable in $NOM_\qqq$ if $\phi$ and $t$ have no free variables in common.
\end{lemma}

\begin{proof}
Refer to Proposition~\ref{4.7}.
\begin{equation*}
\drv
{\drv
{\drv
{\drv
{}
{\Gamma, (\forall x) \phi \Proves (\forall x) \phi}}
{\Gamma, (\forall x) \phi \Proves \phi[x/t]}
&\drv
{}
{\Gamma,(\forall x) \phi \Proves (\forall x) \phi}}
{\Gamma, (\forall x) \phi, \phi[x/t] \Proves (\forall x) \phi}
&\drv
{\drv
{}
{\Gamma, \phi[x/t], (\forall x) \phi \Proves (\forall x) \phi}}
{\Gamma, \phi[x/t], (\forall x) \phi \Proves \phi[x/t]}}
{\Gamma \Proves (\forall x) \phi \Comp \phi[x/t].}
\end{equation*}
\end{proof}

\begin{proposition}\label{5.7}
The following inference rules are derivable in $NOM_\Q$:
\begin{equation*}
\begin{aligned}
\drv
{\Gamma \Proves \phi[x/t]}
{\Gamma \Proves (\exists x)\phi}
\end{aligned}\;,
\qquad \qquad
\begin{aligned}
\drv
{\Gamma \Proves (\exists x)\phi & \Gamma, \phi \Proves \psi & \Gamma, \psi, \phi \Proves \psi }
{\Gamma \Proves \psi}
\end{aligned}\;,
\end{equation*}
where the latter rule is subject to the standard constraint that the variable $x$ must not appear freely in $\Gamma$ or in $\psi$. Both rules are also derivable in $NOM_\qqq$ if the former rule is subject to the additional constraint that $\phi$ and $t$ have no free variables in common.
\end{proposition}

\begin{proof}
Refer to Propositions \ref{2.4}, \ref{4.2}, \ref{4.7}, \ref{4.8}, and \ref{4.9} and Lemma~\ref{5.6}.
\begin{equation*}
\drv
{\drv
{\Gamma \Proves \phi[x/t]
&\drv
{\drv
{\drv
{\drv
{}
{}}
{\Gamma \Proves (\forall x) \Not \phi \Comp  \Not \phi[x/t]}}
{\Gamma \Proves (\forall x) \Not \phi \Comp \phi[x/t]}}
{\Gamma, \phi[x/t], (\forall x) \Not \phi \Proves \phi[x/t]}}
{\Gamma, (\forall x) \Not \phi \Proves \phi[x/t]}
&\drv
{\drv
{}
{\Gamma, (\forall x) \Not \phi \Proves (\forall x) \Not \phi}}
{\Gamma, (\forall x) \Not \phi \Proves \Not \phi[x/t]}}
{\Gamma \Proves \Not (\forall x) \Not \phi.}
\end{equation*}

\begin{equation*}
\drv
{\Gamma \Proves \Not (\forall x) \Not \phi
&\drv
{\drv
{\drv
{\drv
{\Gamma \Proves \Not (\forall x) \Not \phi
&
\drv
{\drv
{\drv
{\drv
{\Gamma, \phi \Proves \psi
&\drv
{}
{\Gamma, \phi \Proves \phi}}
{\Gamma, \phi, \psi \Proves \phi}
& \Gamma, \psi, \phi \Proves \psi}
{\Gamma \Proves \phi \Comp \psi}
& \Gamma, \phi \Proves \psi}
{\Gamma, \Not \psi \Proves \Not \phi}}
{\Gamma, \Not \psi \Proves (\forall x) \Not \phi}}
{\Gamma, \Not (\forall x) \Not \phi, \Not \psi \Proves (\forall x) \Not \phi}}
{\Gamma, \Not (\forall x) \Not \phi, \Not \psi \Proves \Not \Not (\forall x) \Not \phi}}
{\Gamma, \Not (\forall x) \Not \phi \Proves \Not \Not \psi }}
{\Gamma, \Not (\forall x) \Not \phi \Proves \psi }}
{\Gamma \Proves \psi.}
\end{equation*}
\end{proof}

\appendix

\section{}\label{A}

Let $\H$ be a Hilbert space of any dimension. For each subspace $A$ of $\H$, let $\overline A$ be the closure of $A$. For each closed subspace $A$ of $\H$, let $[A]\: \H \to \H$ be the corresponding orthogonal projection operator.

\begin{proposition}\label{A.1}
Let $A$ and $B$ be closed subspaces of $\H$. Then, $A \Andthen B = \overline{[B]A}$.
\end{proposition}

\begin{proof}
Let $\xi \in A$. We calculate that for all $\eta \in A^\perp \wedge B$,
$
\langle \eta | [B] \xi \rangle = \langle [B] \eta | \xi \rangle = \langle \eta | \xi \rangle = 0. 
$
Thus, $[B]\xi \in (A^\perp \wedge B)^\perp = A \vee B^\perp$, and moreover, $[B]\xi \in (A \vee B^\perp) \wedge B = A \Andthen B$. We conclude that $[B]A$ is a subspace of $A \Andthen B$, and therefore, $\overline{[B]A} \subsetof A \Andthen B$.

Let $\zeta \in ([B]A)^\perp$. We calculate that for all $\xi \in A$,
$
\langle \xi | [B] \zeta \rangle = \langle [B] \xi | \zeta \rangle = 0.
$
Thus, $[B] \zeta \in A^\perp$, and moreover, $[B]\zeta \in A^\perp \wedge B$. Of course $(1-[B])\zeta \in B^\perp$, and hence $\zeta = [B] \zeta + (1-[B]) \zeta \in (A^\perp \wedge B) \vee B^\perp = ((A \vee B^\perp) \wedge B)^\perp = (A \Andthen B)^\perp$. We conclude that $([B]A)^\perp \subsetof  (A \Andthen B)^\perp$, and therefore, $A \Andthen B = (A\Andthen B)^{\perp\perp}\subsetof ([B]A)^{\perp\perp} = \overline{[B]A}$. Altogether, we have that $A \Andthen B = \overline{[B]A}$.
\end{proof}

\begin{lemma}\label{A.2}
For all closed subspaces $A_1, \ldots, A_n$, and $B$ of $\H$, we have that $A_1 \Andthen \cdots \Andthen A_n \leq B$ is equivalent to $[A_n]\cdots [A_1]\H \subsetof B$.
\end{lemma}

\begin{proof}
In the case $n = 0$, the equivalence holds as a consequence of the convention that $A_1 \Andthen \cdots \Andthen A_n = \H$. In the case $n =1$, the equivalence holds because $[A_1]\H = A_1$. For the remaining cases, we argue by induction. Let $n \geq 1$, and assume that for all closed subspaces $A_1, \ldots, A_n$, and $B$ of $\H$, we have that $A_1 \Andthen \cdots \Andthen A_n \leq B$ is equivalent to $[A_n]\cdots [A_1]\H \subsetof B$. Then, for all closed subspaces $A_1, \ldots, A_n, A_{n+1}$, and $B$ of $\H$, we reason that 
\begin{align*} &
A_1 \Andthen \cdots \Andthen A_n \Andthen A_{n+1} \leq B
\EV
A_1 \Andthen \cdots \Andthen A_n \leq A_{n+1} \Implies B
\\ & \EV
[A_n] \cdots [A_1] \H \subsetof A_{n+1} \Implies B
\EV
\overline{[A_n] \cdots [A_1] \H} \leq A_{n+1} \Implies B
\\ & \EV
\overline{[A_n] \cdots [A_1] \H} \Andthen A_{n+1} \leq  B
\EV
\overline{[A_{n+1}] \overline{[A_n] \cdots [A_1] \H}} \leq B
\\ & \EV
\overline{[A_{n+1}] [A_n] \cdots [A_1] \H} \leq B
\EV
[A_{n+1}] [A_n] \cdots [A_1] \H \subsetof B,
\end{align*}
where the first and fourth equivalences follow by the adjunction between the Sasaki projection and the Sasaki arrow \cite{Finch} and the fifth equivalence follows by Proposition~\ref{A.1}.
Therefore, by induction on $n$, $A_1 \Andthen \cdots \Andthen A_n \leq B$ if and only if $[A_n]\cdots [A_1]\H \subsetof B$ for all closed subspaces $A_1, \ldots, A_n$, and $B$ of $\H$.
\end{proof}

\begin{theorem}\label{A.3}
Let $A_1, \ldots, A_n$, and $B$ be closed subspaces of $\H$. Then, the following are equivalent:
\begin{enumerate}
\item if a physical system that is modeled by $\H$ is prepared in any initial state and the propositions that are modeled by $A_1, \ldots, A_n$ are measured to be true in that order, then the proposition that is modeled by $B$ is true with probability one;
\item $A_1 \Andthen \cdots \Andthen A_n \leq B$.
\end{enumerate}
\end{theorem}

\begin{proof}
If the physical system is in vector state $\xi$ and the proposition modeled by $A$ is verified to be true, then after this measurement, the physical system is in vector state $\xi ' = (\|[A]\xi\|)\inv [A] \xi$. The scalar $\|[A]\xi\|$ is the square-root of the probability that the proposition modeled by $A$ true in vector state $\xi$, and thus, it is necessarily nonzero if that proposition can be verified.

Let $\xi_0$ be the initial vector state of the system. For each $i \in \{1, \ldots, n\}$, let $p_i$ be the probability that the propositions modeled by $A_1, \ldots, A_i$ can be verified in that order, and let $\xi_i$ be the vector state of the system after such a verification. By a straightforward inductive argument, $p_i = \|[A_i] \cdots [A_1] \xi_0\|^2$ and $\xi_i = p_i^{-1/2}[A_i] \cdots [A_1] \xi_0$ if $p_i \neq 0$. In particular, $p_n = \|[A_n] \cdots [A_1] \xi_0\|^2$ and $\xi_n = p_n^{-1/2}[A_n] \cdots [A_1] \xi_0$ if $p_n \neq 0$. Thus, condition (1) is equivalent to the condition that for each vector state $\xi_0$, if $\|[A_n] \cdots [A_1]\xi_0\|^2 \neq 0$, then $\|[A_n] \cdots [A_1]\xi_0\|\inv [A_n] \cdots [A_1]\xi_0 \in B$.

We obtain a sequence of equivalent conditions:
\begin{itemize}
\item condition (1);
\item for each unit vector $\xi_0 \in \H$, if $\|[A_n] \cdots [A_1]\xi_0\|^2 \neq 0$, then \begin{equation*}\|[A_n] \cdots [A_1]\xi_0\|\inv [A_n] \cdots [A_1]\xi_0 \in B;\end{equation*}
\item for each unit vector $\xi_0 \in \H$, if $\|[A_n] \cdots [A_1]\xi_0\|^2 \neq 0$, then $[A_n] \cdots [A_1]\xi_0 \in B;$
\item for each unit vector $\xi_0 \in \H$, $[A_n] \cdots [A_1]\xi_0 \in B;$
\item $[A_n] \cdots [A_1]\H \subsetof B$;
\item condition (2).
\end{itemize}
The last equivalence follows by Lemma~\ref{A.2}. Thus, the theorem is proved.
\end{proof}

\section{}\label{B}

In this appendix, as in \cite{discretequantumstructures}, we write $\Rel(\X_1, \ldots, \X_n)$ for the complete orthomodular lattice of all binary relations $\X_1 \times \cdots \times \X_n \to \mathbf 1$, where $\X_1, \ldots, \X_n$ are quantum sets and $(\cat{qRel}, \times, \mathbf 1)$ is the symmetric monoidal category of quantum sets and binary relations. We also write $\[\cdot\]_\qqq$ for the interpretation of nonduplicating formulas and terms as it is defined in \cite{discretequantumstructures}.

\begin{lemma}\label{B.1}
Let $\X$, $\Y$, and $\Z$ be quantum sets, let $F$ be a function $\X \to \Y$, and let $R$ be a binary relation $\Y \times \Z \to \mathbf 1$. Then,
\begin{align*}&
\sup\{P \in \Rel (\X) \suchthat P \times \top_\Z \leq R \circ (F \times I_\Z) \}
=
\sup\{Q \in \Rel(\Y) \suchthat Q \times \top_\Z \leq R\} \circ F .
\end{align*}
\end{lemma}

\begin{proof}
Let $S_1 = \sup\{P \in \Rel (\X) \suchthat P \times \top_\Z \leq R \circ (F \times I_\Z) \}$, and let $S_2 = \sup\{Q \in \Rel(\Y) \suchthat Q \times \top_\Z \leq R\}$. We are to prove that $S_1 = S_2 \circ F$. By the definition of $S_2$, we have that $S_2 \times \top_\Z \leq R$, and thus, we also have that $(S_2 \circ F) \times \top_\Z = (S_2 \times \top_\Z) \circ (F \times I_\Z) \leq R \circ (F \times I_\Z)$. Therefore, $S_2 \circ F \leq S_1$ by the definition of $S_1$. Similarly, by the definition of $S_1$, we have that $S_1 \times \top_\Z \leq R \circ (F \times I_\Z)$, and thus, we also have that $(S_1 \circ F^\dagger) \times \top_\Z = (S_1 \times \top_\Z) \circ (F^\dagger \times I_\Z) \leq R \circ (F \times I_\Z) \circ (F^\dagger \times I_\Z) = R \circ ((F \circ F^\dagger) \times I_\Z) \leq R \circ (I_\Y \times I_\Z) = R$ because $F \circ F^\dagger \leq I_\Y$ by the definition of a function between quantum sets. Hence $S_1 \circ F^\dagger \leq S_2$ by the definition of $S_2$. Therefore, $S_1 = S_1 \circ I_\X \leq S_1 \circ F^\dagger \circ F \leq S_2 \circ F$ because $I_\X \leq F^\dagger \circ F$ by the definition of a function between quantum sets. Altogether, we have that $S_1 = S_2 \circ F$, as desired.
\end{proof}

We make two observations for the next proof: First, \cite{discretequantumstructures}*{Prop.~2.3.3}, which relates $\[\overline x \suchthat \phi(\overline x)\]$ to $\[\overline x' \suchthat \phi(\overline x)\]$ when the variables of $\overline x$ all appear in $\overline x'$, applies to nonduplicating formulas as well as to primitive formulas. Indeed, a nonduplicating formula is interpreted by first translating it into a primitive formula \cite{discretequantumstructures}*{sec.~2.7}. Second, $\[\overline x \suchthat t(\overline x)\]$ is a function for any term $t(\overline x)$ by a simple inductive argument that applies \cite{discretequantumstructures}*{Lem.~3.5.3} at each step.

\begin{proposition}\label{B.2}
Let $t(x_1, \ldots, x_n)$ be a nonduplicating term, and let $\phi(y, z_1, \ldots, z_m)$ be a nonduplicating formula. Let $\X_1, \ldots, \X_n$ be the sorts of $x_1, \ldots, x_n$, respectively, let $\Y$ be the sort of $y$, and let $\Z_1, \ldots, \Z_m$ be the sorts of $z_1, \ldots, z_m$, respectively. 
Let $\overline x = (x_1, \ldots, x_n)$, and let $\overline z = (z_1, \ldots, z_m)$; let $\X = \X_1 \times \cdots \times \X_n$, and let $\Z = \Z_1 \times \cdots \times \Z_m$.
Assume that the sort of $t(x_1, \ldots, x_n)$ is $\Y$ and that $\overline x$ and $\overline z$ have no variables in common.
Then,
\begin{equation*}\[\overline x, \overline z \suchthat \phi(t(x_1, \ldots, x_n), z_1, \ldots, z_m)\]_\qqq = \[y, \overline z \suchthat \phi(y, z_1, \ldots, z_m)\]_\qqq \circ (\[\overline x \suchthat t(x_1, \ldots, x_n)\]_\qqq \times I_{\Z}).\end{equation*}
\end{proposition}

\begin{proof}
We prove the proposition by induction on the formula $\phi(y, \overline z)$ for a fixed term $t(\overline x)$. The base case, when $\phi(y, \overline z)$ is an atomic formula, follows from Lemmas 3.5.2 and 3.5.3 of \cite{discretequantumstructures} with Propositions 2.3.3 and 3.1.1 of \cite{discretequantumstructures} serving to handle the expansion and permutation of contexts. The proof consists of tedious bookkeeping that differs only superficially from the bookkeeping for classical predicate logic, and it is omitted.

If $\phi(y, \overline z)$ is of the form $\Not \psi(y, \overline z)$, then we argue that
\begin{align*}
\[\overline x, \overline z \suchthat & \Not \psi(t(\overline x), \overline z)\]_\qqq
= 
\Not \[\overline x, \overline z \suchthat \psi(t(\overline x), \overline z)\]_\qqq
=
\Not (\[y, \overline z \suchthat \psi(y, \overline z)\]_\qqq \circ (\[\overline x \suchthat t(\overline x)\]_\qqq \times I_{\Z}))
\\ & =
\Not \[y, \overline z \suchthat \psi(y, \overline z)\]_\qqq \circ (\[\overline x \suchthat t(\overline x)\]_\qqq \times I_{\Z})
=
\[y, \overline z \suchthat \Not \psi(y, \overline z)\]_\qqq \circ (\[\overline x \suchthat t(\overline x)\]_\qqq \times I_{\Z}),
\end{align*}
where the second equality follows by the induction hypothesis and the third equality follows by \cite{quantumsets}*{Thm.~B.8}. We argue likewise if $\phi(y,\overline z)$ is of the form $\psi_1(y, \overline z) \And \psi_2(y, \overline z)$, if it is of the form $\psi_1(y, \overline z) \Or \psi_2(y, \overline z)$, or if it is of the form $\psi_1(y, \overline z) \Implies \psi_2(y, \overline z)$.

If $\phi(y, \overline z)$ is of the form $(\forall z_{m+1})\psi(y, \overline z, z_{m+1})$ for some variable $z_{m+1}$ of some sort $\Z_{m+1}$, then we argue that
\begin{align*}&
\[\overline x, \overline z \suchthat (\forall z_{m+1})\psi(t(\overline x), \overline z, z_{m+1})\]_\qqq
\\ & = 
\sup\{
P \in \Rel(\X_1, \ldots, \X_n, \Z_1, \ldots, \Z_m)
\suchthat P \times \top_{\Z_{m+1}} \leq \[\overline x, \overline z, z_{m+1} \suchthat \psi(t(\overline x), \overline z, z_{m+1})\]_\qqq
\}
\\ & = 
\sup\{
P \in \Rel(\X_1, \ldots, \X_n, \Z_1, \ldots, \Z_m) 
\\ & \hspace{15ex}
\suchthat P \times \top_{\Z_{m+1}}\leq \[ y, \overline z, z_{m+1} \suchthat \psi(y, \overline z, z_{m+1})\]_\qqq \circ (\[\overline x \suchthat t(\overline x)\]_\qqq \times I_\Z \times I_{\Z_{m+1}})\}
\\ & = 
\sup\{
Q \in \Rel(\Y, \Z_1, \ldots, \Z_m)
\suchthat Q \times \top_{\Z_{m+1}}\leq \[ y, \overline z, z_{m+1} \suchthat \psi(y, \overline z, z_{m+1})\]_\qqq\} 
\\ & \hspace{64ex}
\circ (\[\overline x \suchthat t(\overline x)\]_\qqq \times I_{\Z})
\\ & =
\[ y, \overline z \suchthat (\forall z_{m+1}) \psi(y, \overline z, z_{m+1})\]_\qqq \circ (\[\overline x \suchthat t(\overline x)\]_\qqq \times I_{\Z}),
\end{align*}
where the second equality follows by the induction hypothesis and the third equality follows by Lemma~\ref{B.1}, which is applied to the function $F = \[\overline x \suchthat t(\overline x)\]_\qqq \times I_{\Z}\: \X \times \Z \to \Y \times \Z$ and the binary relation $R = \[ y, \overline z, z_{m+1} \suchthat \psi(y, \overline z, z_{m+1})\]_\qqq \: \X \times \Z \times \Z_{m+1} \to \mathbf 1$. We argue likewise if $\phi(y, \overline z)$ is of the form $(\exists z_{m+1})\psi(y, \overline z, z_{m+1})$.

By induction over the class of all nonduplicating formulas $\phi(y,\overline z)$ for arbitrary tuples $\overline z$, we conclude that the claimed equality holds universally.
\end{proof}

\section*{Acknowledgement}

I thank John Harding for his comments on orthomodular lattices.

\end{document}